%% file: main.tex
\documentclass[11pt, reqno]{amsart}
\usepackage{graphicx} % Required for inserting images
\usepackage{geometry}
\usepackage{lipsum}

\usepackage{url}
\usepackage{amsmath}
\usepackage{amssymb}
\usepackage{amsfonts}
\usepackage{amsthm}
\usepackage{hyperref}
\usepackage{bm}
\usepackage{blkarray}
\usepackage{xcolor}
\usepackage{leftindex}
\usepackage[capitalise]{cleveref}
\usepackage{tikz}
\usepackage{pgfplots}
\pgfplotsset{compat=1.18}
\usepgfplotslibrary{fillbetween}
\usepackage[url=false, isbn=false, doi=false]{biblatex}
\usepackage{fullpage}
\usepackage{caption}

\newtheorem{theorem}{Theorem}[section]
\newtheorem{corollary}[theorem]{Corollary}
\newtheorem{lemma}[theorem]{Lemma}
\newtheorem{remark}[theorem]{Remark}

\theoremstyle{definition}
\newtheorem{definition}[theorem]{Definition}
\theoremstyle{remark}
\newtheorem*{acknowledgement}{Acknowledgement}

\patchcmd{\@startsection}
  {\@afterindenttrue}
  {\@afterindentfalse}
  {}{}

\title{Future stability of large-data wave maps in energy-supercritical dimensions}

\author{Andras Bonk}

\address{Universit\"at Wien, Fakult\"at f\"ur Mathematik,
  Oskar-Morgenstern-Platz 1, 1090 Vienna, Austria}
\email{andras.bonk@univie.ac.at}
\author{Roland Donninger}

\address{Universit\"at Wien, Fakult\"at f\"ur Mathematik,
  Oskar-Morgenstern-Platz 1, 1090 Vienna, Austria}
\email{roland.donninger@univie.ac.at}
  
\thanks{This research was funded in whole or in part by the
Austrian Science Fund (FWF) 10.55776/P34560,
10.55776/PIN2161424, and 10.55776/PAT9429324. For open access purposes, the authors have
applied a CC BY public copyright license to any author-accepted manuscript version arising from this submission.}

\begin{document}

\input{commands}

\input{abstract}

\maketitle

\input{Introduction}

\input{FreeEvo}

\input{LinearEvo}

\input{NonlinearEvo}

\input{proofofmainresults}

\begin{acknowledgement}
    The first author thanks his research group, in particular Frederick Moscatelli for many interesting and insightful discussions and Matthias Ostermann for useful comments on an earlier draft of this paper.
\end{acknowledgement}
\clearpage

\printbibliography

\end{document}

%% file: commands.tex
\newcommand{\abs}[1]{\lvert {#1} \rvert}
\newcommand{\norm}[2]{\left\lVert {#1} \right\rVert_{#2}}
\newcommand{\fracpi}[0]{\frac{\pi}{2}}
\newcommand{\innerp}[3]{({#1} \mid {#2})_{\mathcal{H}_{#3}}}
\newcommand{\japbra}[1]{\langle #1 \rangle}
\newcommand{\unitball}[0]{\mathbb{B}_1^d}
\newcommand{\hypergeometricf}[1]{\leftindex_2 F_1 \left(#1 \right)}
\newcommand{\innprod}[3]{\left( #1 \mid #2 \right)_{#3}}
\newcommand{\sobspace}[1]{H^{#1}\left(\B^d\right)}
\newcommand{\sobspacerad}[1]{H_{\operatorname{rad}}^{#1}\left(\B^d\right)}
\newcommand{\sobspacew}[1]{H^{#1}\left(\B^d;W\right)}
\newcommand{\sobspacewrad}[1]{H_{\operatorname{rad}}^{#1}\left(\B^d;W\right)}
\newcommand{\hilspace}[1]{\mathcal{H}^{#1}(\B^d)}
\newcommand{\hilspacew}[1]{\mathcal{H}^{#1}(\B^d;\mathbf{W})}
\newcommand{\hilspacerad}[1]{\mathcal{H}_{\operatorname{rad}}^{#1}(\B^d)}
\newcommand{\hilspacewrad}[1]{\mathcal{H}_{\operatorname{rad}}^{#1}(\B^d;\mathbf{W})}
\newcommand{\R}[0]{\mathbb{R}}
\renewcommand{\S}[0]{\mathbb{S}}
\newcommand{\C}[0]{\mathbb{C}}
\newcommand{\N}[0]{\mathbb{N}}
\newcommand{\B}[0]{\mathbb{B}}
\newcommand{\pt}[0]{\partial}
\newcommand{\dimregn}[0]{\N_{\operatorname{odd},\geq 5}\times \N_{> \frac{d}{2}}}
\newcommand{\cinftyradw}[0]{C^\infty_{\operatorname{rad}}(\overline{\B^d};W)^2}
\newcommand{\cinftyjointlyw}[0]{C^\infty([0,\infty)\times \overline{\B^d};W)^2}
\newcommand{\cinftyradwone}[0]{C^\infty_{\operatorname{rad}}(\overline{\B^d};W)}
\newcommand{\cinftyjointlywone}[0]{C^\infty([0,\infty)\times \overline{\B^d};W)}
\newcommand{\fixptspace}[1]{\mathfrak{X}^{d,#1}}

%% file: abstract.tex
\begin{abstract}
    We consider energy-supercritical co-rotational wave maps from Minkowski spacetime to the sphere in odd spatial dimensions. The equation admits an explicit co-rotational self-similar blowup solution, which also induces solutions that blow up in the past. In the region after the blowup the solution treated in this paper is remarkable, as it is smooth forward in time and exhibits less than dispersive decay. We prove nonlinear asymptotic stability of this large-data self-similar solution inside forward light cones. In particular, we identify an open set of initial data close to the explicit solution that give rise to forward-in-time wave maps whose decay is slower than that of generic free waves.
\end{abstract}

%% file: Introduction.tex
\section{Introduction}
In this paper, we study \textit{wave maps} with domain manifold $(1+n)$-dimensional Minkowski spacetime $\R^{1,n}$ and target manifold the $n$-sphere $\mathbb{S}^n \subset \R^{n+1}$, embedded as a Riemannian submanifold in Euclidean space $\R^{n+1}$. In this context, the extrinsic formulation of a wave map is a function $U: \R^{1,n} \rightarrow \mathbb{S}^n \subset \R^{n+1}$ that solves the \textit{wave maps equation}
\begin{align}\label{wavemaps}
    \partial^\mu \partial_\mu U + \left(\partial^\mu U \cdot \partial_\mu U \right)U=0.
\end{align}
Here and throughout we follow the conventions that Greek indices run from $0$ to $n$, that $\partial^0= -\partial_0, \partial^i= \partial_i$ for all $i=1,\dots,n$, that Einstein's summation convention is in effect, and that $\cdot$ denotes the Euclidean inner product on $\R^{n+1}$.

Wave maps are subject to a scaling symmetry. If $U$ solves Eq.~\eqref{wavemaps}, then so does $U_\lambda$ given by $U_\lambda(t,x):= U\left(\frac{t}{\lambda},\frac{x}{\lambda}\right)$ for $\lambda >0$. The conserved \textit{energy} is given by
\begin{align*}
    E(U)(t)= \frac{1}{2}\int_{\R^n} \partial_t U(t,\cdot) \cdot \partial_t U(t,\cdot) + \partial^i U(t,\cdot) \cdot \partial_i U(t,\cdot).
\end{align*}
It scales as $E(U_\lambda)(t)=\lambda^{n-2}E(U)\left(\frac{t}{\lambda}\right)$ for $\lambda >0$. This scaling suggests general heuristics for local and global well-posedness in the energy space. Accordingly, the respective models in different dimensions are termed \textit{(energy-)}\textit{subcritical} if $n=1$, \textit{critical} if $n=2$, and \textit{supercritical} if $n \geq 3$. 

We are interested in the supercritical regime. Here the scaling favours the shrinking of large solutions, by which one anticipates blowup. Indeed, over the last few decades, supercritical results have mainly concerned the existence and stability of solutions that blow up in finite time. An important (though not exclusive \cite{ghouldiffblowup}) blowup mechanism is provided by \textit{self-similar} solutions. It is believed that this type of blowup is generic \cite{bizonbiernat,bizonchmajtabor,Donningergenblow}. A wave map is called self-similar if it is invariant under scaling, i.e. if $U_{\lambda}=U$.
\subsection{Explicit solutions and symmetry reduction}
Self-similar solutions already arise within the subclass of \textit{co-rotational} or \textit{one-equivariant} solutions; see ~\cite[p.~109]{geobook} for the general definition. For our purposes, a map $U: \R^{1,n} \rightarrow \S^n \subset \R^{n+1}$ is called co-rotational if there exists a function $\widetilde{u}: \R \times [0,\infty) \rightarrow \R$ such that
\begin{align*}
    U(t,x)=\begin{pmatrix}
        \sin\left(\abs{x}\widetilde{u}(t,\abs{x})\right)\frac{x}{\abs{x}}\\
        \cos\left(\abs{x}\widetilde{u}(t,\abs{x})\right)
    \end{pmatrix}.
\end{align*}
Notably, there are self-similar co-rotational blowup solutions that are known in closed form. By time-translation symmetry, these are given by the family \cite{closedformn3,bizonbiernat}
\begin{align*}
    U_T^*(t,x)=U^*(T-t,x) \ \ \ \text{where} \ \ \ U^*(t,x)=\frac{1}{(n-2 )t^2+\abs{x}^2}\begin{pmatrix}
        2 \sqrt{n-2} tx \\ (n-2)t^2-\abs{x}^2
    \end{pmatrix},
\end{align*}
for any blowup time $T > 0$. Their profile is given by
\begin{align}\label{corotationalprofile}
    \widetilde{u}^*(t,r)=\frac{2}{r}\arctan\left(\frac{r}{\sqrt{n-2}t}\right)
\end{align}
for $t \geq 0, r \geq 0$.
The point is that $U_T^*$ is smooth on the backward light cone of $(T,0)$ apart from the point $(T,0)$ itself where the gradient blows up. 
\subsection{Solutions after blowup.} 
The co-rotational profile $\widetilde{u}$ of a wave map solves
\begin{align*}
    \left(\partial_t^2 - \partial_r^2 - \frac{n+1}{r}\partial_r\right) \widetilde{u}(t,r) + (n-1)\frac{\sin\left(2r \widetilde{u}(t,r)\right)-2r \widetilde{u}(t,r)}{2r^3}=0,
\end{align*}
for $t \in \R, r \geq 0.$
We note that this is the radial counterpart of a semilinear wave equation on $\R^{1,n+2}$. Indeed, from now on let $d:=n+2$, and consider $u: \R \times \R^d \rightarrow \R$ given by $u(t,x)=\widetilde{u}(t,\abs{x})$. We get 
\begin{align}\label{geowaveeqorig}
    \left(\partial_t^2 - \Delta_x\right) u(t,x)+ (d-3)\frac{\sin(2\abs{x} u(t,x))-2\abs{x} u(t,x)}{2\abs{x}^3}=0.
\end{align}
By time-reversal symmetry of the equation, each sign in
\begin{align*}
    \widetilde{u}^*_{\pm}(t,r):=\widetilde{u}^*(\pm t,r) 
\end{align*}
yields a solution. While the blowup problem concerns the behaviour of $U^*_T$ as $t \rightarrow T^{-}$ (or equivalently of $\widetilde{u}^*$ as $t \rightarrow 0^{+}$) \cite{Donningergenblow,DonSchAich,CosDonGlo,CosDonXia,BieDonSch,DonWal23,DonWal25,donninger2025stableselfsimilarblowupcorotational,Glogicrestriction,donninger2026modestabilityselfsimilarwave,donninger2026blowupstabilitywavemaps}, we are instead interested in the stability of $\widetilde{u}^*$ for $t \rightarrow \infty$ (which corresponds to $t \rightarrow -\infty$ for $U^*_T$).

To avoid confusion, we emphasise that although $U^*$ extends to a smooth wave map on all of spacetime apart from the spacetime origin, there is no profile that reflects this behaviour. This is due to the co-rotational ansatz introducing a singularity on some 1-dimensional submanifold of spacetime. A natural candidate, found by exploiting arctangent identities, such as $\widetilde{u}_{\operatorname{cont}}^*(t,r)= \frac{4}{r} \arctan\left(\frac{r}{\sqrt{n-2}t + \sqrt{(n-2)t^2+ r^2}}\right)$, still exhibits singular behaviour as $t < 0$ and $r \rightarrow 0$.

Considering the solution only in the direction of time in which it does not blow up is of interest in its own right. Proving stability in a suitable sense indicates that it plays a significant role in the dynamics of supercritical wave maps. In the context of dynamics this particular solution is remarkable, since $u^*$ does not blow up for $t \geq 1$ and decays linearly as $t \rightarrow \infty$. In contrast, generic solutions to the free equation decay at rate $\simeq t^{-\frac{d-1}{2}}$.

\subsection{Perturbations}
While in the subcritical and critical regime a wider range of tools can be employed, in the supercritical case a perturbative approach appears to be the only currently viable method. We know from above that $u^*(t,x)=\frac{2}{\abs{x}}\arctan\left(\frac{\abs{x}}{\sqrt{d-4}t}\right)$ solves Eq.~\eqref{geowaveeqorig}.
The operation on the left-hand side of Eq.~\eqref{geowaveeqorig} can be reformulated into one that sustains self-similarity by considering $w(t,x):=t u(t,x)$. The equation for $w$ becomes
\begin{align}\label{geowaveeq}
    \overline{\square}w(t,x)+ (d-3) \frac{\sin\left(2 \frac{\abs{x}}{t}w(t,x)\right)-2\frac{\abs{x}}{t}w(t,x)}{2\left(\frac{\abs{x}}{t}\right)^3}=0,
\end{align}
where $\overline{\square}:=t^2 (\partial_t^2-\Delta_x)- 2t\partial_t+2$. The explicit solution $w^*(t,x):= t u^*(t,x)=2\frac{t}{\abs{x}}\arctan{\left(\frac{\abs{x}}{\sqrt{d-4}t}\right)}$ is now self-similar.

We linearise Eq.~\eqref{geowaveeq} around the explicit solution via $w=w^*+\phi$. Denote the nonlinearity by $F(t,x,z):= (d-3)\frac{\sin\left(2\frac{\abs{x}}{t} z\right)-2\frac{\abs{x}}{t} z}{2\left(\frac{\abs{x}}{t}\right)^3}$, $F'(t,x,z):= \partial_z F(t,x,z)$ and decompose the nonlinearity into linear parts and nonlinear remainder
\begin{align*}
    F(t,x,w^*(t,x)+\phi(t,x))=F(t,x,w^*(t,x))+F'(t,x,w^*(t,x))\phi(t,x)+\mathcal{N}(\phi)(t,x),
\end{align*}
where 
\begin{align*}
    \mathcal{N}(\phi)(t,x)&:=F(t,x,w^*(t,x)+\phi(t,x))-F(t,x,w^*(t,x))-F'(t,x,w^*(t,x))\phi(t,x).
\end{align*}
A short computation reveals
\begin{align*}
    \mathcal{V}(t,x):=F'(t,x,w^*(t,x))= \frac{-8(d-3)(d-4)t^4}{((d-4)t^2+\abs{x}^2)^2}.
\end{align*}
The fact that $w^*$ solves Eq.~\eqref{geowaveeq} yields the equation for the perturbation $\phi$
\begin{align}\label{geowaveeqdecomp}
    0=\left(\overline{\square}+\mathcal{V}(t,x)\right)\phi(t,x)+\mathcal{N}(\phi)(t,x).
\end{align}
\subsection{Coordinates and geometry} 
We introduce a hyperboloidal foliation that enables us to define function spaces tailored to the evolution by prescribing a certain behaviour at lightlike infinity. This turns out to be well compatible with the conformal symmetry of the linearised equation. 
\begin{definition}
    The \textit{forward hyperboloidal similarity coordinates (FHSC)} \cite[p. 11]{anilkelv} are given by the map
\begin{align*}
    \Psi: [0,\infty) \times \mathbb{B}^d &\rightarrow \mathbb{R}^{1,d},\\
    (s,y) &\mapsto \left(\frac{e^s}{1-\abs{y}^2},\frac{e^s y}{1-\abs{y}^2}\right).
\end{align*}
\end{definition}
Onto its image $D^+(\Sigma_0^d)$, $\Psi$ is a diffeomorphism with inverse $\Psi^{-1}(t,x)= (-\log(\frac{t}{t^2-\abs{x}^2}),\frac{x}{t})$. It is the composition $K \circ \chi$ of classical similarity coordinates 
\begin{align*}
    \chi: [0,\infty) \times \mathbb{B}^d &\rightarrow \Gamma,\\
    (\tau, \xi) &\mapsto (-e^{-\tau},e^{-\tau}\xi),
\end{align*}
where $\Gamma:=\left\{(t,x) \in \mathbb{R}^{1,d}: \abs{x}< \abs{t}, -1 \leq t < 0 \right\}$ is the past light cone of the origin truncated at $-1$ and the Kelvin transform
\begin{align*}
    K: \Gamma &\rightarrow D^+(\Sigma_0^d),\\
    (t,x) &\mapsto \left(\frac{-t}{t^2-\abs{x}^2},\frac{x}{t^2-\abs{x}^2}\right).
\end{align*}
Here $D^+(\Sigma_0^d)$ is the subset of the interior of the future light cone emerging from $\left(\frac{1}{2},0\right) \in \R^{1,d}$ foliated by the hyperboloidal surfaces $\Sigma_s^d := \{\Psi(s,y): y \in \B^d \}$ for $s \geq 0$. In fact, $D^+(\Sigma_0^d)$ represents the future domain of dependence of $\Sigma_0^d$. Note that $\Sigma_0^d$ asymptotically approaches the boundary of the light cone with vertex $\left(\frac{1}{2},0\right)$. We invite the reader to consult \cref{fhscpic} for an illustration of the coordinates.
\input{coordinatepic}
\begin{remark}
    One advantage of the foliation is that it allows for favourable energy estimates. Solutions to the free wave equation can propagate energy through lightlike infinity; hence, norms computed over hyperboloidal leaves of the foliation may decay in time. This is in contrast to Cartesian coordinates, where the energy is conserved. For a more detailed discussion about the intricacies of hyperboloidal coordinates and initial value problems, we refer to \cite{anilscrifix} and \cite{anilkelv}.
\end{remark}
\begin{definition}
    For $\ell \in \{1,3\}$ we consider the weight
    \begin{align*}
       W_\ell: \B^d &\rightarrow \R,\\
       y &\mapsto (1-|y|^2)^\frac{\ell-d}{2},
    \end{align*}
    and for every $s \geq 0$ the function space
    \begin{align*}
        \mathcal{C}_{\operatorname{rad}}^\infty(\Sigma_s^d, W_1):= \{f \mid W_1 f \circ \Psi(s,\cdot) \in C_{\operatorname{rad}}^\infty(\overline{\B^d})\} 
    \end{align*}
    and for every $k \in \N$ the weighted Sobolev norms for $\mathbf{f}= (f_1,f_2) \in \mathcal{C}_{\operatorname{rad}}^\infty(\Sigma_s^d, W_1)^2$ given by
    \begin{align*}
        \norm{\mathbf{f}}{\mathcal{H}^{k}(\Sigma_s^d)} = \sqrt{\norm{W_1 f_1 \circ \Psi(s,\cdot)}{H^k(\B^d)}^2 + \norm{W_1 f_2 \circ \Psi(s,\cdot)}{H^{k-1}(\B^d)}^2}.
    \end{align*}
\end{definition}
In Cartesian coordinates $W_\ell$ corresponds to a weight singular at lightlike infinity, namely $W_\ell \circ \Psi^{-1} (t,x)= (\frac{t^2}{t^2-\abs{x}^2})^\frac{d-\ell}{2}$.

\subsection{Main results}
Here we state the stability results. The implications are discussed afterwards.
\begin{theorem}\label{mainthmwavemaps}
    Let $n, k \in \N$ with
    \begin{align*}
        n \geq 3 \ \text{odd} \ \ \text{and} \ \ k > n/2 + 1.
    \end{align*}
    There exists $\varepsilon > 0$ such that for all smooth, co-rotational data $(F_1,F_2): \Sigma_0^n \rightarrow \emph{T}\S^n \subset \R^{n+1}\times \R^{n+1}$ of the form
    \begin{align*}
        F_1(t,x) = \begin{pmatrix}
            \sin(|x| (\widetilde{u}^*+\widetilde{f_1})(t,|x|)) \frac{x}{|x|} \\
            \cos(|x| (\widetilde{u}^*+\widetilde{f_1})(t,|x|)) 
        \end{pmatrix},
        F_2(t,x) = \begin{pmatrix}
            \cos(|x|(\widetilde{u}^*+\widetilde{f}_1)(t,|x|))(\pt_0 \widetilde{u}^*+\widetilde{f}_2)(t,|x|)x\\
            -|x| \sin(|x| (\widetilde{u}^*+\widetilde{f}_1)(t,|x|))(\pt_0 \widetilde{u}^*+\widetilde{f}_2)(t,|x|)
        \end{pmatrix}
    \end{align*}
    with $\widetilde{f}_1,\widetilde{f_2} \in C_{\operatorname{rad}}^\infty(\Sigma_0^1)$ which satisfy $\widetilde{f}_i(t,\abs{x})=f_i(t,x)$ for $i\in \{1,2\}$ with
    \begin{align*}
        \norm{(f_1,f_2)}{\mathcal{H}^k(\Sigma_0^{n+2})} \leq \varepsilon,
    \end{align*}
there exists a unique co-rotational $U \in C^\infty(D^+(\Sigma_0^n))$ that solves the initial value problem
\begin{align*}
    \begin{cases}
    \pt^\mu \pt_\mu U + (\pt^\mu U \cdot \pt_\mu U)U=0 \ \ &\text{in} \ D^+(\Sigma_0)\setminus \Sigma_0,\\
    U = F_1,\  \pt_0 U = F_2 &\text{in} \ \Sigma_0,
    \end{cases}
    \end{align*} 
    and whose profile $\widetilde{u}$ can be decomposed as $\widetilde{u}=\widetilde{u}^*+\widetilde{\phi}$ such that
    \begin{align*}
        \sup_{s \geq 0}e^{(1+\delta) s}\norm{(\phi,\omega \pt_0 \phi)}{\mathcal{H}^k\left(\Sigma_s^{n+2}\right)} \lesssim \norm{(f_1,f_2)}{\mathcal{H}^k\left(\Sigma_0^{n+2}\right)},
    \end{align*}
    where $\omega(t,x)= \frac{t^2-|x|^2}{t}$ and $\phi(t,x)=\widetilde{\phi}(t,\abs{x})$. Fix $x \in \R^n$ and consider $(t,x) \in D^+(\Sigma_0^n)$. We have
    \begin{align*}
        \left|U(t, x)- \begin{pmatrix}
            0 \\ 1
        \end{pmatrix}\right| \simeq t^{-1}
    \end{align*}
    for all $t >0$ such that $(t,x) \in D^+(\Sigma_0^n)$, where $\begin{pmatrix}
        0 \\ 1
    \end{pmatrix}$ denotes the north pole of $\S^n$.
\end{theorem}

This theorem is a direct consequence of the following result for the semilinear wave equation Eq.~\eqref{geowaveeqorig} induced by the profile of $U$. From now on, we abbreviate $\Sigma_s^d$ as $\Sigma_s$.
\begin{theorem}\label{mainthm}
    Let $d,k \in \N$ with
    \begin{align*}
        d \geq 5 \ \text{odd} \ \ \text{and} \ \ k > d/2.
    \end{align*}
    There exists $\varepsilon >0$ such that for every $(f_1, f_2) \in \mathcal{C}_{\operatorname{rad}}^\infty(\Sigma_0,W_1)^2$ with $\norm{(f_1,f_2)}{\mathcal{H}^k\left(\Sigma_0\right)} \leq \varepsilon$, the hyperboloidal initial value problem
    \begin{align*}
        \begin{cases}
            0 = \square u(t,x)+ (d-3)\frac{\sin\left(2\abs{x}u(t,x)\right)-2\abs{x}u(t,x)}{2\abs{x}^3} \ \ &\text{in} \ D^+(\Sigma_0)\setminus \Sigma_0,\\
            u = u^*+f_1 &\text{in} \ \Sigma_0,\\
            \pt_0 u =\pt_0 u^*+f_2 &\text{in} \ \Sigma_0,
        \end{cases}
    \end{align*}
    has a unique solution $u \in C^\infty(D^{+}(\Sigma_0))$. The solution can be decomposed as $u=u^* + \phi$, where $\phi$ has the following properties:
    \begin{itemize}
        \item $\phi \circ \Psi \in C^\infty([0,\infty) \times \overline{\B^d}; W_1)$,
        \item $\phi \circ \Psi(s,\cdot)$ is radial for all $s \geq 0$.
    \end{itemize}
    The solution is attracted by $u^*$ in the sense that there exists $\delta >0$ such that
    \begin{align*}
        \sup_{s \geq 0}\left(e^{(1+\delta) s}\left(\norm{\phi}{H^k\left(\Sigma_s\right)}+\norm{\omega \pt_0 \phi}{H^{k-1}\left(\Sigma_s\right)}\right)\right) \lesssim \norm{f_1}{H^k\left(\Sigma_0\right)}+ \norm{f_2}{H^{k-1}\left(\Sigma_0\right)}.
    \end{align*}
    In particular, for fixed $x \in \R^d$, we have
    \begin{align*}
        u(t,x) \simeq t^{-1}
    \end{align*}
    for all $t >0$ such that $(t,x) \in D^+(\Sigma_0)$.
\end{theorem}
\subsection{Discussion}
\begin{remark}
     The stability result suggests the following (highly simplified) picture of the dynamics. We order the solution space according to the asymptotic behaviour as $t \rightarrow \infty$ of the solutions. Starting from the trivial solution, there are global solutions that decay to $0$ as $t \rightarrow \infty$. Then, numerical evidence (at least for $3 \leq n \leq 6$) indicates that there is the ``threshold'' solution given by the wave map corresponding to the first excited state; see \cite{bizonchmajtabor} and \cite{Bizonthreshold3} for $n=3$ and \cite{biernatbizonmaliborski} for $4 \leq n \leq 6$. Beyond this lies the blowup solution given by $U^*_T$, which is believed to describe generic blowup. Zooming in on the global solutions, first one finds solutions with dispersive decay and then solutions that decay more slowly. This is where the time-reversed blowup solution considered in our analysis sits. In general we expect that this is not be the only time-reversed state that plays a role there. Apart from our result concerning $U^*$ the rest remains conjectural.
     
     Note, however, that these dynamics cannot be ordered by the size of the data, as is evidenced by the observation that only the sign in $u^*_{\pm}$ dictates whether the solution blows up or exists for all positive times. This does not change the size of the data as it only changes the sign of the second initial datum.
\end{remark}
\subsection{Proof strategy}    
If $w_\Psi \in C^\infty([0,\infty) \times \mathbb{B}^d)$ is related to $w \in C^\infty(D^+(\Sigma_0))$ via $w_\Psi= w \circ \Psi$, then Eq.~\eqref{geowaveeq} transforms into
\begin{align}\label{geowaveeqsesim}
    0=\widehat{\square} w_\Psi (s,y)+ F(\Psi(s,y),w_\Psi (s,y)),
\end{align}
where $\widehat{\square}$ is such that $\widehat{\square} (\cdot \circ \Psi)=\overline{\square} \cdot \circ \Psi$, i.e.
\begin{align*}
    &\widehat{\square} w_\Psi(s,y)= \left(\partial_s^2 + 2 \partial_s \Lambda_y + \left(y^k y^\ell - \delta^{k\ell}\right)\partial_k \partial_\ell
    + \frac{2d-3- 3\abs{y}^2}{1- \abs{y}^2}\partial_s + 4\Lambda_y +2 \right)w_\Psi(s,y),
\end{align*}
where $\Lambda_y:= y^k\partial_{y^k}$. Also
\begin{align*}
    F(\Psi (s,y),w_\Psi (s,y))=(d-3) \frac{\sin \left(2\abs{y}w_\Psi(s,y)\right)-2 \abs{y}w_\Psi(s,y)}{2\abs{y}^3}.
\end{align*}
According to the decomposition in Eq.~\eqref{geowaveeqdecomp} with $\phi_\Psi:= \phi \circ \Psi, w^*_\Psi:= w^* \circ \Psi$ we get 
\begin{align}\label{geowaveeqsesimdecomp}
    0=\widehat{\square} \phi_\Psi(s,y)+ F'(\Psi(s,y), w^*_\Psi (s,y)) \phi_\Psi(s,y)+ \mathcal{N}_\Psi(\phi_ \Psi)(s,y),
\end{align}
where 
\begin{align*}
    \mathcal{N}_\psi(\phi_\Psi)(s,y)=F(\Psi(s,y),(w^*_\Psi+\phi_\Psi)(s,y))-F(\Psi(s,y),w^*_\Psi(s,y))-F'(\Psi(s,y),w^*_\Psi(s,y))\phi_\Psi(s,y).
\end{align*}
Note that $w_\Psi^*(s,y)=\frac{2}{\abs{y}}\arctan\left(\frac{\abs{y}}{\sqrt{d-4}}\right)$ is independent of time $s$. We abbreviate by $w^*_\Psi(y):=w_\Psi^*(s,y)$ for all $s \geq 0$. Explicitly, Eq.~\eqref{geowaveeqsesimdecomp} is
\begin{align}\label{fulleq}
    0=\left(\widehat{\square}+V(y)\right) \phi_\Psi(s,y)+ N(\phi_\Psi(s,\cdot))(y),
\end{align}
where
\begin{align*}
    \begin{split}
        &\widehat{\square}v(s,y)=\left[\partial_s^2+2 \partial_s \Lambda_y+ (y^k y^\ell- \delta^{k\ell}) \partial_k \partial_\ell + \frac{2d-3-3\abs{y}^2}{1-\abs{y}^2}\partial_s +4\Lambda_y+2\right] v(s,y),\\
        &V(y)v(s,y):=-\frac{8(d-3)(d-4)}{(d-4+\abs{y}^2)^2}v(s,y),\\
        &N(v(s,\cdot))(y):=(d-3)\frac{\sin\left(2 \abs{y}(w^*_\Psi(y)+v(s,y))\right)-2\abs{y}v(s,y)-\sin\left(2 \abs{y} w^*_\Psi(y)\right)}{2\abs{y}^3}-V(y)v(s,y).
    \end{split}
\end{align*}
Considering the abstract Cauchy problem induced by Eq.~\eqref{fulleq} we prove the following result.

Set $W:= W_3$. We introduce the space $\hilspacewrad{k}$ which is the completion of $\cinftyradw$ with respect to the norm given by $\norm{\mathbf{f}}{\hilspacew{k}}^2= \norm{\mathbf{W}\mathbf{f}}{\hilspace{k}}^2= \norm{W f_1}{\sobspace{k}}^2+ \norm{W f_2}{\sobspace{k-1}}^2.$
\begin{theorem}\label{mainthmfhsc}
    Let $d,k$ be as in \Cref{mainthm}. There exists $\varepsilon >0$ such that there exists a unique mild solution $\boldsymbol{\phi} \in C\left([0,\infty), \hilspacewrad{k}\right)$ to the abstract Cauchy problem (ACP)
    \begin{align*}
        \begin{cases}
            \pt_s \boldsymbol{\phi}(s)&=\mathbf{L}\boldsymbol{\phi}(s)+ \mathbf{N}(\boldsymbol{\phi}(s)) \ \text{for all} \ s >0,\\
            \boldsymbol{\phi}(0)&= \mathbf{f},
        \end{cases}
    \end{align*}
    for all $\mathbf{f} \in \hilspacewrad{k}$ with $\norm{\mathbf{f}}{\hilspacew{k}} \leq \varepsilon$. Additionally, we have the following decay estimate. There exists $\delta > 0$ such that
    \begin{align*}
        \sup_{s \geq 0} e^{\delta s}\norm{\boldsymbol{\phi}(s)}{\hilspacew{k}} \lesssim \norm{\mathbf{f}}{\hilspacew{k}}.
    \end{align*}
    Assuming additionally that $\mathbf{f} \in \mathcal{D}(\mathbf{L})$, $\boldsymbol{\phi} \in C\left([0,\infty), \hilspacewrad{k}\right) \cap C^1\left((0,\infty), \hilspacewrad{k}\right)$ is the unique classical solution to (ACP).
\end{theorem}
We explain the ingredients of \Cref{mainthmfhsc}. Simultaneously this outlines the paper.
\begin{itemize}
    \item (ACP) is constructed such that $\mathbf{L}= \mathbf{L}_0+ \mathbf{L}'$ where $\mathbf{L}_0$ generates the free wave flow (corresponding to $\widehat{\square}$) and $\mathbf{L}'$ is the operator analogue of the potential arising from the linearisation of the nonlinearity around the explicit solution (corresponding to $V$). $\mathbf{N}$ is the remaining nonlinear term (corresponding to $N$).
    \item The main point is that we can prove decay estimates for the flow generated by $\mathbf{L}_0$ by connecting it to a related free wave problem via the conformal involution. This correspondence leads to the introduction of weighted function spaces in which the decay is evident. The reciprocal of the singular spatial weight induced by the conformal inversion is smooth only in the case of odd dimension, which is the technical reason why we restrict ourselves to this assumption.
    \item The introduction of $\mathbf{L}'$ to consider the flow generated by $\mathbf{L}$ then poses no real obstacle anymore which is surprising since usually this is the difficult part in the semigroup argument. After gaining deeper insight into the spectral structure of $\mathbf{L}$ we find that only isolated eigenvalues with non-negative real part could hinder the decay of the linearised wave flow. In the related blowup problem the following spectral analysis is divided into two regions. On the one hand the region between the imaginary axis and the eigenvalue of the mode induced by time-translation symmetry (which has positive real part in the related blowup problem) and on the other hand the half-plane right of that eigenvalue. The half-plane right of the mode eigenvalue is handled by a standard Sturm-Liouville argument. The region left of it translates to solving a \textit{connection problem} for which there is no general strategy known yet and usually corresponds to the difficult part in the analysis. In our case however, the eigenvalue induced by time-translation symmetry lies left of the imaginary axis (which is due to time-translation having no effect on the asymptotics) which means that the Sturm-Liouville argument suffices to show nonexistence of eigenvalues with nonnegative real part.
    \item Then the full nonlinear problem is treated by a standard fixed-point argument using the decay from the linearised flow.
    \item The last section covers the implications of \Cref{mainthmfhsc}, most importantly its connection to \Cref{mainthm}.
\end{itemize}

\subsection{Related results}
As one of the prototypical examples of a geometric wave equation, the wave maps equation has attracted a lot of attention over the last few decades. Since the literature is vast and impossible to review in its entirety, we restrict ourselves to long-time dynamics of related PDE mostly in connection with self-similar solutions and mention blowup dynamics in passing. For an elaborate exposition of the theory of dispersive PDE we refer to \cite{klainermansurvey,taobook}, for an exposition of wave maps in particular \cite{Kriegersurvey}, and the more recent \cite{gebabook}. 

The existence of a self-similar blowup solution to the supercritical wave maps equation was proved in \cite{shatahexofblowup} and found in closed form in \cite{closedformn3} for $n=3$ and in \cite{bizonbiernat} for $n \geq 4$. The theoretical problem concerning large data which do not lead to blowup is mostly uncharted waters for supercritical wave maps. A notable exception in the case of one-equivariant (or co-rotational) wave maps is \cite{kriegerconstr} where, following the procedure initiated for the related problem for the nonlinear wave equation with power nonlinearity (NLW) \cite{kriegerconstrpower}, smooth global solutions are constructed which are stable under certain perturbations. The stability of the blowup solution $u^*$ after the blowup for $t \rightarrow \infty$ inside the light cone emerging from the singularity is however not included in this project since the construction assumes a smallness condition along the time axis that is not satisfied by the time-reversed blowup solution. Another avenue was taken in \cite{germainbesov}, where the evolution is tracked in Besov spaces which accommodate the singular behaviour of the blowup solution. Numerically, the threshold separating blowup solutions from global ones, which corresponds to considering the first excited state, has been the content of \cite{biernatbizonmaliborski} in dimensions $3 \leq n \leq 6$. 

For the 3-dimensional cubic wave equation the hyperboloidal stability problem after blowup in the case of radial symmetry has been the content of \cite{bizonanil} numerically, and of \cite{anilkelv} theoretically, for results outside of radial symmetry see \cite{BurtscherDonninger} and for the initial value problem in the classical sense \cite{duyckaerts24}. The central idea in \cite{bizonanil} and \cite{anilkelv} was to exploit the conformal symmetry of the (linearised) equation to connect the forward stability to the blowup stability. Although in case of the cubic NLW the whole equation enjoys conformal symmetry, the symmetry of the linearised equation suffices for the argument. We take a similar approach. A difference is that our equation is translated into an at least $5$-dimensional wave equation, which leads to the introduction of dimension-dependent weights singular at null infinity in contrast to the $3$-dimensional case of the cubic NLW. At the same time this leads to a dimension-induced decay-in-time property. As a result, in this paper the nonlinear stability is of codimension $0$ in contrast to the cubic wave equation case where it is stated as a codimension $4$ stability. The stability in \cite{anilkelv} is stated under perturbations such that the full solution is in the same space in which it is perturbed in. A consequence and main difference of the dimension-induced singular weights is that the function space in which the perturbations are tracked does not include the explicit blowup solution itself.

\subsection{Notation}
The sets of natural, real and complex numbers are denoted by $\N$, $\R$, $\C$ respectively. The $d$-dimensional open ball of radius $1$ is given by $\B^d= \{x\in \R^d \mid \|x\| < 1\}$.

If $a_{\iota}, b_{\iota} \geq 0$ are nonnegative real numbers indexed by the parameter $\iota \in I$, we define the relation $a_{\iota} \lesssim b_{\iota}$ if there exists a constant $C > 0$ such that $a_\iota \leq C b_\iota$ uniformly for all $\iota \in I$. Accordingly, $a_\iota \simeq b_\iota$ iff $a_\iota \lesssim b_\iota$ and $b_\iota \lesssim a_\iota$.

For a function space $X(\B^d)$ consisting of functions defined on $\B^d$ and a weight $w$ nonzero on $\B^d$ we write $X(\B^d;w)$ for the space which is obtained as $w^{-1}X(\B^d)= \{f \mid wf \in X(\B^d)\}$. If $X(\B^d)$ is a normed vector space, the norm $\|\cdot \|_{X(\B^d)}$ induces a norm on $X(\B^d;w)$ by $\norm{f}{X(\B^d;w)}= \norm{w f}{X(\B^d)}$. To avoid confusion, as an example, we mention that the weighted $L^2$-space obtained by multiplying the measure by $w$, sometimes denoted by $L^2(\B^d, w(r) \mathrm{d}r)$, would correspond to $L^2(\B^d; \sqrt{w})$ in our notation. 

We write $C^\infty(\overline{\B^d})$ for the space of smooth functions whose partial derivatives of all orders are bounded on $\B^d$. We abuse notation to also write $C^\infty([0,1])$, $C^\infty([0,\infty))$, $C^\infty([0,\infty)\times \B^d)$ and $C^\infty([0,\infty)\times \overline{\B^d})$. The subspace of radial functions is given by 
\begin{align*}
    C^\infty_{\operatorname{rad}}(\overline{\B^d}):=\{f \in C^\infty(\overline{\B^d}) \mid f \text{ is radially symmetric}\}.
\end{align*}
and the space of smooth even functions on the closed unit interval is given by
\begin{align*}
    C^\infty_{\operatorname{ev}}([0,1]):=\{f \in C^\infty([0,1]) \mid \forall n \in \N : \lim_{x \rightarrow 0^+} f^{(2n-1)}(x)=0\}.
\end{align*}
Recall from \cite{ostermannradial} that for any $f \in C^\infty_{\operatorname{rad}}(\overline{\B^d})$ there is a radial profile $\widetilde{f} \in C^\infty_{\operatorname{ev}}([0,1])$ such that $f(x) = \widetilde{f}(\abs{x})$.

For $k \in \N_0$, we introduce the Sobolev norm of $f \in C^\infty(\overline{\B^d})$ as
\begin{align*}
    \norm{f}{H^k(\B^d)}:= (\sum_{0 \leq \abs{\alpha} \leq k} \norm{\partial^\alpha f}{L^2(\B^d)}^2)^\frac{1}{2}.
\end{align*}
The space $C^\infty(\overline{\B^d})$ endowed with the norm $\norm{\cdot}{H^k(\B^d)}$ is an inner product space whose completion is the Sobolev space $H^k(\B^d)$.

For function spaces $X_1, X_2$ we consider the product space $X_1 \times X_2$ consisting of tuples of functions 
\begin{align*}
    \mathbf{f}= (f_1,f_2)= \begin{bmatrix}
        f_1 \\ f_2
    \end{bmatrix},
\end{align*}
where $f_1 \in X_1$, $f_2 \in X_2$. Here and throughout the paper we use boldface letters for elements of product spaces. If $X_1, X_2$ are normed spaces then so is $X_1 \times X_2$. For compatibility with inner product structure the norm on the product space is always induced by $\norm{\mathbf{f}}{X_1\times X_2}:= \sqrt{\norm{f_1}{X_1}^2+\norm{f_2}{X_2}^2}$.

%% file: coordinatepic.tex
\begin{figure}
    \centering
    \begin{tikzpicture}
    \begin{axis}[
        axis lines=middle,
        xlabel={$|x|$},
        ylabel={$t$},
        xmin=0, xmax=5,
        ymin=0, ymax=5,
        ticks=none
    ]
        \addplot[red, very thick, domain=0:5, samples=200] 
            {(1+sqrt(4*x^2+1))/2};
        \addlegendentry{$l(y)$}
        
        \addplot[black, dashed, domain=0:5, samples=200] 
            {0.5 + sqrt(x^2)};
        \addlegendentry{$\frac{1}{2} + \sqrt{y^2}$}
        
        \addplot[black, domain=0:5, samples=400] 
            {1.4*abs(x) > 2*x^2/(-1 + sqrt(1 + 4*x^2)) ? 1.4*abs(x) : nan};
        \addlegendentry{$1.4|y|$}
        
        \addplot[black, domain=0:5, samples=200] 
            {1.7*abs(x) > 2*x^2/(-1 + sqrt(1 + 4*x^2)) ? 1.7*abs(x) : nan};
        \addlegendentry{$1.7|y|$}
        
        \addplot[black, domain=0:5, samples=200] 
            {2.5*abs(x) > 2*x^2/(-1 + sqrt(1 + 4*x^2)) ? 2.5*abs(x) : nan};
        \addlegendentry{$2.5|y|$}
        
        \addplot[black, domain=0:5, samples=600] 
            {5*abs(x) > 2*x^2/(-1 + sqrt(1 + 4*x^2)) ? 5*abs(x) : nan};
        \addlegendentry{$5|y|$}
        
        \addplot[black, domain=0:5, samples=200] 
            {(sqrt(1 + 4*x^2*0.4^2)+ 1)/(2*0.4)};
        \addlegendentry{$k(y, 0.4)$}
        
        \addplot[black, domain=0:5, samples=200] 
            {(sqrt(1 + 4*x^2*0.55^2) + 1)/(2*0.55};
        \addlegendentry{$k(y, 0.55)$}
        
        \addplot[black, domain=0:5, samples=200] 
            {(sqrt(1 + 4*x^2*0.75^2) + 1)/(2*0.75)};
        \addlegendentry{$k(y, 0.75)$}
        
        \addplot[black, domain=0:5, samples=200] 
            {(1+sqrt(1+4*0.3^2*x^2))/(2*0.3)};
        \addlegendentry{$k(y, 0.3)$}

        \addplot[black, domain=0:5, samples=200] 
            {(sqrt(1 + 4*x^2*0.23^2) + 1)/(2*0.23)};
        \addlegendentry{$k(y, 0.9)$}
        
        \legend{};
    \end{axis}
\end{tikzpicture}
\caption{Projection of FHSC. The straight lines correspond to $y = \text{const}$, the hyperboloids correspond to $s = \text{const}$. The red hyperboloid $\Sigma_0$ is the lower boundary of the coordinate region and functions as the initial hypersurface for the IVP. The dashed line depicts the light cone with vertex $\left(\frac{1}{2},0 \right)$.}
    \label[figure]{fhscpic}
\end{figure}
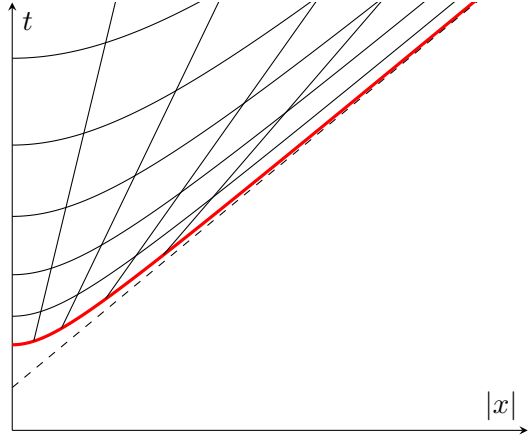

%% file: FreeEvo.tex
\section{Free wave evolution in FHSC}

In the following, we implement a functional-analytic framework for the flow associated with the rescaled wave operator in similarity coordinates
\begin{align}\label{freeevow}
    \widehat{\square}w(s,y)=\left[\partial_s^2+2 \partial_s \Lambda_y+ (y^k y^\ell- \delta^{k\ell}) \partial_k \partial_\ell + \frac{2d-3-3\abs{y}^2}{1-\abs{y}^2}\partial_s +4\Lambda_y+2\right] w(s,y).
\end{align}
\subsection{Connection to the blowup problem}
We make the critical observation that the variable 
\begin{align}\label{functiontransform}
    f(s,y):=e^{(d-2)s}(1-\abs{y}^2)^{\frac{3-d}{2}} w(s,y)
\end{align}
yields the differential operator
\begin{align}\label{freebweq}
    \frac{e^{(d-2)s}}{(1-\abs{y}^2)^\frac{d-3}{2}}\widehat{\square}w(s,y)= \left[\partial_s^2+2\partial_s \Lambda_y-(\delta^{k \ell}-y^k y^\ell)\partial_{y^k} \partial_{y^\ell}+\partial_s+2 \Lambda_y \right]f(s,y)=:\square^*f(s,y),
\end{align}
which, as a consequence of the conformal symmetry in physical coordinates, for example appears in the free evolution of a related blowup problem \cite{donninger2025stableselfsimilarblowupcorotational}. This suggests tracking the evolution in function spaces compatible with the conformal symmetry.

First, let us recall the functional-analytic set-up used in the blowup problem. By introducing the components $f_1(s,y):= f(s,y), f_2(s,y):=(\partial_s+\Lambda_y)f(s,y)$, Eq.~\eqref{freebweq} is recast into the first order system
\begin{align}\label{firstordersystemstar}
    \partial_s \begin{bmatrix}
        f_1(s,y)\\ f_2(s,y)
    \end{bmatrix}= \begin{bmatrix}
        f_2(s,y)-\Lambda_y f_1(s,y)\\ -\Lambda_y f_2(s,y)-f_2(s,y)+\Delta_y f_1(s,y)
    \end{bmatrix} + \begin{bmatrix}
        0 \\ \square^* f(s,y)
    \end{bmatrix}.
\end{align}
This leads to the following definition.
\begin{definition}
    Let $d \in \mathbb{N}$, $(g_1, g_2)  \in C^\infty(\overline{\mathbb{B}^d})^2$ and define
    \begin{align*}
        \mathbf{L}^*_0\begin{bmatrix}
            g_1\\ g_2
        \end{bmatrix}(y):= \begin{bmatrix}
            g_2(y)-\Lambda g_1(y) \\ -\Lambda g_2(y)- g_2(y)+ \Delta g_1(y)
        \end{bmatrix}.
    \end{align*}
\end{definition}
To establish the connection to the evolution associated with the differential operator $\widehat{\square}$ from Eq.~\eqref{freeevow} we define the components 
\begin{align}\label{variablepair}
    &w_1(s,y):= w(s,y), &w_2(s,y):= \left(\partial_s+\Lambda_y+ (d-2)+ (d-3)\frac{\abs{y}^2}{1-\abs{y}^2}\right)w(s,y).
\end{align}
Define $W(y):=(1-\abs{y}^2)^\frac{3-d}{2}$, the spatial part of the conformal weight from the start of the section, and set $\mathbf{W}\mathbf{h}:=(W h_1, W h_2)$ and $\mathbf{W}^{-1}\mathbf{h}:=(\frac{1}{W} h_1, \frac{1}{W} h_2)$ for general $\mathbf{h}=(h_1,h_2)$. The pair $(w_1,w_2)$ is chosen such that
\begin{align}\label{connecttoblowupeq}
    e^{(2-d)s} \mathbf{W}^{-1}\begin{bmatrix}
        f_1(s,\cdot) \\ f_2(s,\cdot)
    \end{bmatrix}
    = \begin{bmatrix}
        w_1(s,\cdot) \\ w_2(s,\cdot)
    \end{bmatrix}.
\end{align}
Notably
\begin{align*}
    \partial_s \begin{bmatrix}
        w_1(s,\cdot) \\ w_2(s,\cdot)
    \end{bmatrix} &= \mathbf{W}^{-1} \partial_s \begin{bmatrix}
        e^{(2-d)s}f_1(s,\cdot) \\ e^{(2-d)s}f_2(s,\cdot)
    \end{bmatrix}\\
    &= \mathbf{W}^{-1}e^{(2-d)s}\left(((2-d)\mathbf{I}+\partial_s)\begin{bmatrix}
        f_1(s,\cdot) \\ f_2(s,\cdot)
    \end{bmatrix}\right)\\
    &=\mathbf{W}^{-1}e^{(2-d)s}\left(((2-d)\mathbf{I}+\mathbf{L}^{*}_0)\begin{bmatrix}
        f_1(s,\cdot) \\ f_2(s,\cdot)
    \end{bmatrix}+
    \begin{bmatrix}
        0 \\ \square^* f(s,\cdot)
    \end{bmatrix}\right)\\
    &= \mathbf{W}^{-1}\mathbf{L}^{*,2-d}_0 \mathbf{W}\begin{bmatrix}
        w_1(s,\cdot) \\ w_2(s,\cdot)
    \end{bmatrix}+
    \begin{bmatrix}
        0 \\ \widehat{\square}w(s,\cdot)
    \end{bmatrix}.
\end{align*}
Hence the operator of interest formally is $\mathbf{L}_0:=\mathbf{W}^{-1}\mathbf{L}^{*,2-d}_0 \mathbf{W}$, where $\mathbf{L}^{*,2-d}_0$ denotes the shifted operator $\mathbf{L}^{*}_0 - (d-2)\mathbf{I}$.

To apply the result from \cite[p.~63, Theorem B.1]{donninger2025stableselfsimilarblowupcorotational}, we only need to set up the corresponding function spaces. Here we assume that the dimension is odd to ensure smoothness of $W$.
\begin{definition}
    Let $3 \leq d \in \N$ be odd and $k \in \mathbb{N}$. Define the space of smooth functions that vanish to order $\frac{d-3}{2}$ at the boundary as
    \begin{align*}
        C^\infty(\overline{\mathbb{B}^d};W) := \left\{f \in C^\infty(\overline{\mathbb{B}^d}): Wf \in C^\infty(\overline{\mathbb{B}^d})\right\}.
    \end{align*}
We equip $\cinftyradw$ with the norm $\norm{\cdot}{\hilspacew{k}}:=\norm{\mathbf{W}\cdot}{\hilspace{k}}$. Denote by $\hilspacew{k}$ the completion of $\cinftyradw$ with respect to $\norm{\cdot}{\hilspacew{k}}$.
Analogously we define the radial subspace
\begin{align*}
    C_\text{rad}^\infty(\overline{\mathbb{B}^d};W):= \left\{f \in C_\text{rad}^\infty(\overline{\mathbb{B}^d}): Wf \in C_\text{rad}^\infty(\overline{\mathbb{B}^d})\right\}.
\end{align*}
The radial subspace $\hilspacewrad{k} \subset \hilspacew{k}$ is obtained as the completion of $\cinftyradw$ with respect to $\norm{\cdot}{\hilspacew{k}}$.
\end{definition}
\begin{remark}
\begin{itemize}
    \item Clearly $\mathbf{W}: (\cinftyradw, \norm{\cdot}{\hilspacew{k}}) \rightarrow (C^\infty(\overline{\mathbb{B}^d})^2, \norm{\cdot}{\hilspace{k}})$ is an isometric homeomorphism between inner product spaces which extends to a unitary operator between $\hilspacew{k}$ and $\hilspace{k}$ and its restriction on radial subspaces. In notation we will not distinguish between $\mathbf{W}$ and its extensions on Sobolev spaces.
    \item $(\hilspacew{k})_{k \in \N}$ is a nested sequence of Hilbert spaces as a consequence of the same being true for $(\hilspace{k})_{k \in \N}$.
    \item Note that there is no obvious relation between the spaces $\hilspacew{k}$ and the weighted product Sobolev spaces obtained by multiplying the measure with any weight. 
\end{itemize}
\end{remark}
\subsection{The free evolution is uniformly exponentially stable}
With these preparations we are now able to state sharp growth estimates for the free evolution in every odd dimension greater than or equal to $3$ and every regularity.
\begin{lemma}\label{freesemigroup}
    Let $3 \leq d \in \mathbb{N}$ be odd and $k \in \mathbb{N}$. The operator $\mathbf{L}_0: \mathcal{D}(\mathbf{L}_0) \subset \hilspacew{k} \rightarrow \hilspacew{k}$, given by
    \begin{align*}
        &\mathbf{L}_0\mathbf{f}=\mathbf{W}^{-1}\mathbf{L}_0^{*,2-d}\mathbf{W}\mathbf{f},&\mathcal{D}(\mathbf{L}_0):=C^\infty(\overline{\mathbb{B}^d};W)^2,
    \end{align*}
    is closable and its closure $\overline{\mathbf{L}_0}:\mathcal{D}\left(\overline{\mathbf{L}_0}\right) \subset \hilspacew{k} \rightarrow \hilspacew{k}$, given by
    \begin{align*}
        &\overline{\mathbf{L}_0}\mathbf{f}=\mathbf{W}^{-1}\overline{\mathbf{L}_0^{*,2-d}}\mathbf{W}\mathbf{f}, &\mathcal{D}\left(\overline{\mathbf{L}_0}\right)=\mathbf{W}^{-1}\mathcal{D}\left(\overline{\mathbf{L}_0^{*}}\right),
    \end{align*}
    generates the strongly continuous semigroup $\mathbf{S}_0:=\mathbf{W}^{-1}\mathbf{S}_0^{*,2-d}(\cdot)\mathbf{W}: [0,\infty) \rightarrow \mathcal{L}\left(\hilspacew{k}\right)$. For every $0<\varepsilon<\frac{1}{2}$ there exists $M_{d,k,\varepsilon} \geq 1$ such that the semigroup satisfies the growth estimate
    \begin{align*}
        &\norm{\mathbf{S}_0(s)\mathbf{f}}{\hilspacew{k}} \leq M_{d,k,\varepsilon} e^{(\omega_{d,k,\varepsilon}-(d-2))s}\norm{\mathbf{f}}{\hilspacew{k}}, &\omega_{d,k,\varepsilon}= \max \left\{\frac{d}{2}-k,\varepsilon \right\}
    \end{align*}
    for all $\mathbf{f} \in \hilspacew{k}$ and for all $s \geq 0$.
    
    The appearing operators are compatible with radial symmetry. Namely, the operator $\mathbf{L}_0 \mathord \restriction: \mathcal{D}(\mathbf{L}_0 \mathord \restriction) \subset \hilspacewrad{k} \rightarrow \hilspacewrad{k}$ given by
    \begin{align*}
        &\mathbf{L}_0 \mathord \restriction \mathbf{f}= \mathbf{L}_0 \mathbf{f}, &\mathcal{D}(\mathbf{L}_0 \mathord \restriction)= C_{\operatorname{rad}}^\infty(\overline{\mathbb{B}^d};W)^2
    \end{align*}
    is closable and its closure $\overline{\mathbf{L}_0 \mathord \restriction}: \mathcal{D}(\overline{\mathbf{L}_0 \mathord \restriction}) \subset \hilspacewrad{k} \rightarrow \hilspacewrad{k}$ given by
    \begin{align*}
        &\overline{\mathbf{L}_0 \mathord \restriction}\mathbf{f}= \overline{\mathbf{L}_0} \mathbf{f}, &\mathcal{D}(\overline{\mathbf{L}_0 \mathord \restriction})= \mathcal{D}(\overline{\mathbf{L}_0}) \cap \hilspacewrad{k},
    \end{align*}
    generates the radial subspace semigroup $\mathbf{S}_0 \mathord \restriction:[0,\infty) \rightarrow \mathcal{L}\left(\hilspacewrad{k}\right)$ given by
    \begin{align*}
        \mathbf{S}_0 \mathord \restriction (s)\mathbf{f}= \mathbf{S}_0 (s)\mathbf{f}
    \end{align*}
    for all $\mathbf{f} \in \hilspacewrad{k}$ and all $s \geq 0$.
\end{lemma}
\begin{proof}
    Note that $\mathbf{L}_0^*$ equals the operator in \cite[p.~62, Theorem B.1]{donninger2025stableselfsimilarblowupcorotational} for $R_0=R=1, h \equiv -1$. $\mathbf{S}_0$ is a semigroup since it is similar to $\mathbf{S}_0^{*,2-d}$ via the homeomorphism $\mathbf{W}$ (for an introduction to similar semigroups we refer to \cite[59]{engnag}). Here $\mathbf{S}_0^{*,2-d}(s):=e^{-(d-2)s}\mathbf{S}_0^*(s)$ is the rescaled semigroup of $\mathbf{S}_0^*$ obtained in \cite{donninger2025stableselfsimilarblowupcorotational}. The generator of $\mathbf{S}_0=\mathbf{W}^{-1}\mathbf{S}_0^{*,2-d}\mathbf{W}$ is given by $\mathbf{W}^{-1}\overline{\mathbf{L}_0^{*,2-d}}\mathbf{W}$ which equals the closure of $\mathbf{L}_0$ for $\mathbf{W}$ is a homeomorphism. The function spaces were set up in such a way that the growth estimate follows directly from the estimate in \cite{donninger2025stableselfsimilarblowupcorotational} and semigroup scaling \cite[60]{engnag} by
    \begin{align*}
        \norm{\mathbf{S}_0(s)\mathbf{f}}{\hilspacew{k}} =\norm{\mathbf{W}^{-1}\mathbf{S}_0^{*,2-d}(s)\mathbf{W}\mathbf{f}}{\hilspacew{k}} \leq M_{d,k,\varepsilon} e^{(\omega_{d,k,\varepsilon}-(d-2))s} \norm{\mathbf{W}\mathbf{f}}{\hilspace{k}}
    \end{align*}
    for all $\mathbf{f} \in \hilspacew{k}$ and all $s \geq 0$.

    The ability to restrict to radial subspaces is a consequence of the same being true in \cite{donninger2025stableselfsimilarblowupcorotational}.
\end{proof}
\begin{remark}
    We point out that the growth bound is sharp, since the conformal involution of $1$ yields (after multiplication by $t$) the admissible solution $e^{-(d-2) s}(1-|y|^2)^{\frac{d-3}{2}}$. 
\end{remark}
\begin{remark}
    Although a priori this is a statement about a pair of generator and semigroup on $\hilspacew{k}$ for each $k \geq 1$ respectively, one should think about the family of semigroups as a single one acting on the largest space $\hilspacew{k_{\operatorname{min}}}$ with additional properties when restricted to subspaces $\hilspacew{\ell}$. This is true in much more generality and is sometimes referred to as ``restriction property" \cite[Lemma C.1]{Glogicrestriction}. It is also true for any family of semigroups and generators that follow and explains why throughout the paper we suppress in notation the dependence of semigroups and generators on the underlying Hilbert spaces.
\end{remark}

%% file: LinearEvo.tex
\section{Linearised evolution}
In this section, we introduce the potential as a bounded perturbation into the evolution. Since this term arises from the nonlinearity induced by supercritical wave maps, we restrict to spatial dimensions $d \geq 5$. As in the previous section, we assume that $d$ is odd. Indeed, throughout this section, let $d \geq 5$ be odd and $k \in \N$.
We are interested in the linearised evolution
\begin{align*}
    (\widehat{\square}+V(y))v(s,y)=0,
\end{align*}
where $V(y)=-\frac{8(d-3)(d-4)}{(d-4+\abs{y}^2)^2}$.

Note that multiplication with $V$ defines a bounded operator  $H_{\operatorname{rad}}^{k}(\mathbb{B}^d;W) \rightarrow H_{\operatorname{rad}}^{k}(\mathbb{B}^d;W)$ for all $k \in \N$, since every derivative of $V$ is bounded on $\mathbb{B}^d$.

\subsection{Generation of linearised semigroup}

We immediately switch to the semigroup picture. 
\begin{definition}\label{linearisedgeneratordef}
    The radial operator $\mathbf{L}:=\overline{\mathbf{L}_0 \mathord \restriction}+\mathbf{L}': \mathcal{D}(\mathbf{L}) \subset \hilspacewrad{k} \rightarrow \hilspacewrad{k}$ is defined by $\mathcal{D}(\mathbf{L}):=\mathcal{D}(\overline{\mathbf{L}_0 \mathord \restriction})$ and
\begin{align*}
    &\mathbf{L}' \in \mathcal{L}(\hilspacewrad{k}), &\mathbf{L}' \begin{bmatrix}
        f_1 \\ f_2
    \end{bmatrix}:= \begin{bmatrix}
        0 \\ -V f_1
    \end{bmatrix}.
\end{align*}
\end{definition}
The boundedness of multiplication with $V$ implies that $\mathbf{L}'$ is bounded and in particular closed. Thus $\mathbf{L}$ is a closed operator. Perturbation theory for semigroups shows that $\mathbf{L}$ generates a linearised wave flow.
\begin{lemma}\label{linearisedevoexist}
     The closed linear operator $\mathbf{L}: \mathcal{D}(\mathbf{L}) \subset \hilspacewrad{k} \rightarrow \hilspacewrad{k}$ is the generator of a strongly continuous one-parameter semigroup $\mathbf{S}: [0,\infty) \rightarrow \mathcal{L}(\hilspacewrad{k})$. Additionally, let $0 < \varepsilon < \frac{1}{2}$ and $M_{d,k,\varepsilon}$, $\omega_{d,k,\varepsilon}$ be the constants from \Cref{freesemigroup}. Then $\mathbf{S}$ satisfies
    \begin{align*}
        \norm{\mathbf{S}(s)\mathbf{f}}{\hilspacew{k}} \leq M_{d,k,\varepsilon} e^{\left(\omega_{d,k,\varepsilon}-(d-2)+M_{d,k,\varepsilon}\norm{\mathbf{L}'}{\mathcal{L}(\hilspacew{k})}\right)s}\norm{\mathbf{f}}{\hilspacew{k}}
    \end{align*}
    for all $\mathbf{f} \in \hilspacewrad{k}$ and $s \geq 0$.
\end{lemma}
\begin{proof}
    The operator $\mathbf{L}$ generating a semigroup $\mathbf{S}$ with the asserted growth estimate readily follows from the bounded perturbation theorem \cite[p.~158]{engnag} applied to the semigroup $\mathbf{S}_0 \mathord \restriction$ from \Cref{freesemigroup}.
\end{proof}
Since we do not have control over $M_{d,k,\varepsilon}\norm{\mathbf{L}'}{\mathcal{L}(\hilspacew{k})}$ the growth estimate is not sufficient to prove a stable evolution. In the following, we will refine the growth estimate by gaining deeper insight into the spectrum of $\mathbf{L}$.

\subsection{Spectral analysis}
For orientation we outline the rest of the section:
\begin{itemize}
    \item First we reduce the linear stability to mode stability by showing that $\mathbf{L}'$ is compact, that the part of the spectrum of $\mathbf{L}$ that would correspond to an unstable evolution is confined to a compact region, and applying the analytic Fredholm theorem.
    \item We then show that mode solutions would induce solutions for a one-dimensional problem.
    \item The occurrence of these solutions is then ruled out by a standard application of Sturm-Liouville theory.
    \item The spectral analysis results in the linear stability of the evolution by an application of the Gearhart-Prüss-Greiner Theorem.
\end{itemize}
The first step is to observe that $\mathbf{L}'$ is not only bounded but even compact, which readily follows from compactness of Sobolev embeddings.
\begin{lemma}\label{compactness}
     The operator $\mathbf{L}': \hilspacewrad{k} \rightarrow \hilspacewrad{k}$ is compact.
\end{lemma}
\begin{proof}
    Decompose $\mathbf{L}'$ into bounded operators
    \begin{align*}
     \mathcal{H}_{\operatorname{rad}}^{k}(&\mathbb{B}^d;\mathbf{W}) \rightarrow \mathcal{H}_{\operatorname{rad}}^{k}(\mathbb{B}^d) \rightarrow H_{\operatorname{rad}}^{k}(\mathbb{B}^d) \hookrightarrow H_{\operatorname{rad}}^{k-1}(\mathbb{B}^d) \rightarrow \hspace{5pt} \mathcal{H}_{\operatorname{rad}}^{k}(\mathbb{B}^d) \hspace{3pt} \rightarrow  \mathcal{H}_{\operatorname{rad}}^{k}(\mathbb{B}^d;\mathbf{W}),\\
    &\begin{bmatrix}
        f_1 \\ f_2
    \end{bmatrix} \hspace{10pt}\mapsto \begin{bmatrix}
        W f_1 \\ W f_2
    \end{bmatrix} \hspace{9pt} \mapsto \hspace{9pt} W f_1 \hspace{10pt}\mapsto \hspace{10pt} W f_1 \hspace{15pt}\mapsto \begin{bmatrix}
        0 \\ -V W f_1
    \end{bmatrix} \mapsto \begin{bmatrix}
        0 \\ -V f_1
    \end{bmatrix}.
\end{align*}
Since the Sobolev embedding in between is compact, $\mathbf{L}'$ is compact as a composition of bounded and compact operators.
\end{proof}
In a next step we show that the part of the spectrum of $\mathbf{L}$ which would correspond to an unstable evolution is confined to a bounded region.

Since we will use an explicit description of $\overline{\mathbf{L}_0\mathord \restriction}$ we preface with a technical lemma.
\begin{lemma}\label{techlem}
     Let $\mathbf{f} \in \mathcal{D}\left(\overline{\mathbf{L}_0 \mathord \restriction}\right)$ and set $\mathbf{g}:=(g_1,g_2):= \overline{\mathbf{L}_0 \mathord \restriction}\mathbf{f}$. Then $\mathbf{f}= (f_1,f_2)$ satisfies
    \begin{align*}
        &g_1=f_2 - W^{-1}\Lambda (W f_1) - (d-2) f_1,\\
        &g_2= -W^{-1}\Lambda(W f_2) - f_2 + W^{-1} \Delta(W f_1)- (d-2)f_2, 
    \end{align*}
    in the sense of distributions.
\end{lemma}
\begin{proof}
Let $\mathbf{f} \in \mathcal{D}\left(\overline{\mathbf{L}_0 \mathord \restriction}\right)$. By definition of the closure, there is a sequence $(\mathbf{f}_n)_{n \in \mathbb{N}} \subset \mathcal{D}(\mathbf{L}_0 \mathord \restriction) = \cinftyradw$ such that $\lim_{n \rightarrow \infty}\norm{\mathbf{f}-\mathbf{f}_n}{\hilspacew{k}}=0$ and $\lim_{n \rightarrow \infty}\norm{\mathbf{g}-\mathbf{g}_n}{\hilspacew{k}}=0$, where $\mathbf{g}_n:= \mathbf{L}_0 \mathord \restriction \mathbf{f}_n$. By definition of $\mathbf{L}_0 \mathord \restriction$ we have
\begin{align*}
    &g_{1n}=f_{2n} - W^{-1}\Lambda (W f_{1n}) - (d-2) f_{1n},\\
    &g_{2n}= -W^{-1}\Lambda(W f_{2n}) - f_{2n} + W^{-1} \Delta(W f_{1n})- (d-2)f_{2n}, 
\end{align*}
where $\mathbf{f}_n=(f_{1n},f_{2n})$ and $\mathbf{g}_n=(g_{1n},g_{2n})$. Testing the equations with a function in $C_c^\infty\left(\mathbb{B}^d\right)$, integrating by parts, and taking the limit $n \rightarrow \infty$ yields the claims.
\end{proof}
\begin{lemma}\label{gammar}
     Let $0<\varepsilon <\frac{1}{2}$ and $\omega_{d,k,\varepsilon}$ be the corresponding constant from \Cref{freesemigroup}. There exists $R > 0$ such that $\sigma(\mathbf{L}) \cap \Gamma_R = \emptyset$, where
    \begin{align*}
        \Gamma_R:= \{z \in \mathbb{C} \mid \operatorname{Re}(z) > \frac{\omega_{d,k,\varepsilon}-(d-2)}{2}, \abs{z}> R\}.
    \end{align*}
    Furthermore, we have
    \begin{align*}
        \sup_{\lambda \in \Gamma_R} \norm{(\lambda \mathbf{I}-\mathbf{L})^{-1}}{\mathcal{L}(\hilspacewrad{k})} < \infty.
    \end{align*}
\end{lemma}
\begin{proof}
    As a consequence of the growth estimate $\norm{\mathbf{S}_0\mathord \restriction (s)}{\hilspacewrad{k}} \leq M_{d,k,\varepsilon} e^{(\omega_{d,k,\varepsilon}-(d-2))s}$ from \Cref{freesemigroup} the integral representation of the free resolvent \cite[p.~55]{engnag}
    \begin{align*}
        \mathbf{R}(\lambda, \overline{\mathbf{L}_0 \mathord \restriction}):= (\lambda \mathbf{I}- \overline{\mathbf{L}_0 \mathord \restriction})^{-1}= \int_0^\infty e^{-\lambda s} \mathbf{S}_0 \mathord \restriction(s) \, \mathrm{d}s
    \end{align*}
    exists for all $\lambda \in \mathbb{C}$ satisfying $\operatorname{Re}(\lambda) > \omega_{d,k,\varepsilon}-(d-2)$. The \textit{Birman-Schwinger principle}, i.e.~the identity
    \begin{align*}
        \lambda \mathbf{I}- \mathbf{L}= (\mathbf{I}- \mathbf{L}'\mathbf{R}(\lambda, \overline{\mathbf{L}_0 \mathord \restriction}))(\lambda \mathbf{I}- \overline{\mathbf{L}_0 \mathord \restriction}),
    \end{align*}
    implies that for $\operatorname{Re}(\lambda) > \omega_{d,k,\varepsilon}-(d-2)$, $\lambda \mathbf{I}- \mathbf{L}$ is bounded invertible if and only if $\mathbf{I}-\mathbf{L}' \mathbf{R}(\lambda, \overline{\mathbf{L}_0})$ is bounded invertible.
    
    We show that $\mathbf{I}-\mathbf{L}' \mathbf{R}(\lambda, \overline{\mathbf{L}_0})$ is bounded invertible for all $\lambda \in \Gamma_R$ for some $R >0$ by a Neumann series argument. Let $\operatorname{Re}(\lambda) > \frac{\omega_{d,k,\varepsilon}-(d-2)}{2}$, $\mathbf{g} \in \hilspacewrad{k}$ and set $\mathbf{R}(\lambda,\overline{\mathbf{L}_0 \mathord \restriction})\mathbf{g}=: \mathbf{f}$. Thus $\mathbf{f} \in \mathcal{D}(\overline{\mathbf{L}_0 \mathord \restriction})$ and $(\lambda \mathbf{I}- \overline{\mathbf{L}_0 \mathord \restriction})\mathbf{f}= \mathbf{g}$. In particular, by \Cref{techlem} we have
    \begin{align*}
        g_1= -f_2 + W^{-1}\Lambda (W f_1) + (d-2+\lambda) f_1,
    \end{align*}
    i.e.
    \begin{align*}
        f_1= \frac{1}{d-2+\lambda} (f_2 - W^{-1}\Lambda (W f_1)+g_1).
    \end{align*}
    Hence 
    \begin{align*}
        \norm{\mathbf{L}' \mathbf{R}(\lambda,\overline{\mathbf{L}_0 \mathord \restriction})\mathbf{g}}{\hilspacew{k}} &= \norm{\begin{pmatrix}
            0\\ -V f_1
        \end{pmatrix}}{\hilspacew{k}} \\
        & \lesssim \frac{1}{\abs{d-2+\lambda}} \left(\norm{W f_2}{H^{k-1}(\mathbb{B}^d)} + \norm{\Lambda (W f_1)}{H^{k-1}(\mathbb{B}^d)} +\norm{W g_1}{H^{k-1}(\mathbb{B}^d)}\right),\\
        &\lesssim \frac{1}{\abs{d-2+\lambda}} (\norm{\mathbf{f}}{\hilspacew{k}}+ \norm{\mathbf{g}}{\hilspacew{k-1}})\\
        &\lesssim \frac{1}{\abs{d-2+\lambda}} \norm{\mathbf{g}}{\hilspacew{k}},
    \end{align*}
    where we used $\norm{\Lambda \cdot}{H^{k-1}(\mathbb{B}^d)} \lesssim \norm{\cdot}{H^{k}(\mathbb{B}^d)}$ in the third and 
    \begin{align*}
        \norm{\mathbf{f}}{\hilspacew{k}}=\norm{\mathbf{R}(\lambda,\overline{\mathbf{L}_0 \mathord \restriction})\mathbf{g}}{\hilspacew{k}} \leq {\frac{M_{d,k,\varepsilon}}{|\operatorname{Re}\lambda - (\omega_{d,k,\varepsilon}-(d-2))|}} \norm{\mathbf{g}}{\hilspacew{k}} \lesssim \norm{\mathbf{g}}{\hilspacew{k}}
    \end{align*} 
    in the fourth step. Note that we have used here that the distance of the real part of $\lambda$ to the growth estimate can be uniformly estimated away from $0$. Choosing $R > 0$ large enough yields $\norm{\mathbf{L}' \mathbf{R}(\lambda,\overline{\mathbf{L}_0 \mathord \restriction})}{\mathcal{L}(\hilspacew{k})} \leq \frac{1}{2}$ for all $\lambda \in \Gamma_R$. Hence $(\mathbf{I}- \mathbf{L}' \mathbf{R}(\lambda,\overline{\mathbf{L}_0 \mathord \restriction}))^{-1}$ exists as a Neumann series and $\norm{(\mathbf{I}- \mathbf{L}' \mathbf{R}(\lambda,\overline{\mathbf{L}_0 \mathord \restriction}))^{-1}}{\mathcal{L}(\hilspacew{k})} \leq 2$ for all $\lambda \in \Gamma_R$. The additional statement follows by
    \begin{align*}
        \sup_{\lambda \in \Gamma_R} \norm{(\lambda \mathbf{I}-\mathbf{L})^{-1}}{\mathcal{L}(\hilspacew{k})} &\leq \sup_{\lambda \in \Gamma_R} \norm{\mathbf{R}(\lambda,\overline{\mathbf{L}_0 \mathord \restriction})}{\mathcal{L}(\hilspacew{k})} \norm{(\mathbf{I}- \mathbf{L}' \mathbf{R}(\lambda,\overline{\mathbf{L}_0 \mathord \restriction}))^{-1}}{\mathcal{L}(\hilspacew{k})} \\
        &\leq {\frac{2 M_{d,k,\varepsilon}}{|\operatorname{Re}\lambda - (\omega_{d,k,\varepsilon}-(d-2))|}} \lesssim 1.
    \end{align*}
\end{proof}
\begin{remark}
    Note that the set $\Gamma_R$ can even be enlarged by choosing a weaker constraint on the distance to the connected part of the spectrum. For simplicity we choose the distance to be $\frac{\omega_{d,k,\varepsilon}-(d-2)}{2}$ but the proof is valid for any positive distance. However, we get sufficient information by using a coarser estimate.
\end{remark}
Combining the previous results we get the reduction to mode stability. That is, since $\mathbf{L}'$ is compact, \Cref{gammar} in conjunction with the \textit{analytic Fredholm theorem} \cite[p.~194]{simonbook} implies that $\lambda \mathbf{I}- \mathbf{L}$ is bounded invertible for all $\lambda \in \mathbb{C}$ with $\operatorname{Re}\lambda > \frac{\omega_{d,k,\varepsilon}-(d-2)}{2}$ except for a finite number of eigenvalues, each with finite algebraic multiplicity. Following conventional terminology, we call the eigenfunctions corresponding to these eigenvalues \textit{mode solutions}.

In order to locate mode solutions, we need to solve $(\lambda \mathbf{I}- \mathbf{L})\mathbf{f}=0.$ We show that potential mode solutions induce smooth solutions to a one-dimensional problem. Afterwards we rule out their occurrence.
\begin{lemma}\label{inducelemma}
     Any nonzero $\mathbf{f} \in \mathcal{D}(\mathbf{L})= \mathcal{D}(\overline{\mathbf{L}_0 \mathord \restriction}) \subset \hilspacewrad{k}$ which satisfies $(\lambda \mathbf{I}- \mathbf{L}) \mathbf{f}=0$ for $\lambda \in \mathbb{C}$ with $\operatorname{Re}(\lambda) > -1$ induces a nonzero smooth $\widetilde{f}_1^W \in C^\infty([0,1])$ which solves 
    \begin{align*}
            ((\cdot)^2-1) (\widetilde{f}_{1}^W)''+\frac{2(\lambda+d-1)(\cdot)^2+1-d}{(\cdot)} (\widetilde{f}_{1}^W)'+\left((\lambda+d-1)(\lambda+d-2)+V\right)\widetilde{f}_{1}^W = 0
    \end{align*}
    and satisfies
    \begin{align*}
        \widetilde{f}_1^W \simeq 1
    \end{align*}
    near both endpoints.
\end{lemma}
\begin{proof}
    By definition of the closure, $\mathbf{f} \in \mathcal{D}\left(\overline{\mathbf{L}_0\mathord \restriction}\right) \subset \hilspacewrad{k}$ means that there exists a sequence $(\mathbf{f}_n)_{n \in \mathbb{N}} \subset \cinftyradw$ such that $\lim_{n \rightarrow \infty}\norm{\mathbf{f}-\mathbf{f}_n}{\hilspacew{k}}=0$ and $\lim_{n \rightarrow \infty}\norm{(\lambda \mathbf{I}-\mathbf{L})\mathbf{f}_n}{\hilspacew{k}}=0$. Denote $\mathbf{g}_n := (\lambda \mathbf{I}- \mathbf{L})\mathbf{f}_n \subset C_{\operatorname{rad}}^\infty\left(\overline{\mathbb{B}^d};W\right)^2$. Set $\lambda_1:= \lambda+d-1,\lambda_2:= \lambda+d-2$. By definition of $\mathbf{L}$ we have
    \begin{align*}
        \begin{cases}
            \lambda_2 f_{1n}-f_{2n}+ W^{-1}\Lambda(W f_{1n})=g_{1n},\\
            \lambda_1 f_{2n}+W^{-1}\Lambda(W f_{2n})-W^{-1}\Delta(W f_{1n})+Vf_{1n}=g_{2n},
        \end{cases}
    \end{align*}
    where $(f_{1n},f_{2n})=\mathbf{f}_n, (g_{1n},g_{2n})=\mathbf{g}_n.$ To simplify computations we use the one-to-one correspondence via $\mathbf{W}$ to consider $(\mathbf{W}\mathbf{f}_n)_{n \in \mathbb{N}}, (\mathbf{W}\mathbf{g}_n)_{n \in \mathbb{N}} \subset C_{\operatorname{rad}}^\infty(\overline{\mathbb{B}^d})^2$ instead. The above equation then turns into
    \begin{align*}
        \begin{cases}
            \lambda_2 f_{1n}^W-f_{2n}^W+ \Lambda f_{1n}^W=g_{1n}^W,\\
            \lambda_1 f_{2n}^W+\Lambda f_{2n}^W-\Delta f_{1n}^W+Vf_{1n}^W=g_{2n}^W,
        \end{cases}
    \end{align*}
    by $(f_{1n}^W,f_{2n}^W):= \mathbf{W}\mathbf{f}_n$ and $ (g_{1n}^W,g_{2n}^W):= \mathbf{W}\mathbf{g}_n$. In terms of the radial representatives, denoted with a tilde, the equation reads
    \begin{align*}
        \begin{cases}
            \lambda_2 \widetilde{f}_{1n}^W-\widetilde{f}_{2n}^W+ (\cdot) (\widetilde{f}_{1n}^{W})'=\widetilde{g}_{1n}^W,\\
            \lambda_1 \widetilde{f}_{2n}^W+(\cdot) (\widetilde{f}_{2n}^W)'-\frac{d-1}{(\cdot)} (\widetilde{f}_{1n}^W)'-(\widetilde{f}_{1n}^W)''+V \widetilde{f}_{1n}^W=\widetilde{g}_{2n}^W.
        \end{cases}
    \end{align*}
    Solving towards $\widetilde{f}_{1n}^W$ leads to
    \begin{align*}
        \begin{cases}
            \lambda_2 \widetilde{f}_{1n}^W+ (\cdot) (\widetilde{f}_{1n}^{W})'-\widetilde{g}_{1n}^W=\widetilde{f}_{2n}^W,\\
            ((\cdot)^2-1)(\widetilde{f}_{1n}^W)''+\frac{2 \lambda_1 (\cdot)^2+1-d}{(\cdot)} (\widetilde{f}_{1n}^W)'+\left(\lambda_1 \lambda_2 +V\right)\widetilde{f}_{1n}^W= (\cdot)(\widetilde{g}_{1n}^W)'+ \lambda_1 \widetilde{g}_{1n}^W+\widetilde{g}_{2n}^W.
        \end{cases}
    \end{align*}
Recall the assumptions on the convergence of $\mathbf{f}_n, \mathbf{g}_n$ in $\hilspacew{k}$ and the equivalence $\norm{h}{H^{k}(\B^d)} \simeq \norm{\abs{\cdot}^{\frac{d-1}{2}}\widetilde{h}}{H^k(0,1)}$ for every $h \in C_{\operatorname{rad}}^\infty(\overline{\mathbb{B}^d})$, see \cite[Lemma B.4]{ostermann21}. This implies that $\widetilde{f}_{1n}^W, \widetilde{g}_{1n}^W$ are Cauchy sequences in $H^k((0,1);(\cdot)^{\frac{d-1}{2}})$ with limiting functions $\widetilde{f}_{1}^W \neq 0,0$ respectively. In particular this implies $\widetilde{f}_{1}^W \in H^k((\delta,1))$ for every $\delta > 0$ which in turn implies distributional convergence. The distributional differential equation is then given by
\begin{align*}
    ((\cdot)^2-1) (\widetilde{f}_{1}^W)''+\frac{2 \lambda_1 (\cdot)^2+1-d}{(\cdot)} (\widetilde{f}_{1}^W)'+\left(\lambda_1 \lambda_2 +V\right)\widetilde{f}_{1}^W=0.
\end{align*}
On $(0,1)$ this an ordinary differential equation with smooth coefficients. Hence the solution is already classical \cite[p.~58]{hormanderbook}, and in a second step smooth.

We turn to the behaviour at the endpoints. We compute the Frobenius indices
\begin{align*}
    \{0, 2-d\} \ \text{at $0$} \hspace{2cm} \text{and} \ \hspace{2cm} \{0,-\lambda-\frac{d-3}{2}\} \ \text{at $1$},
\end{align*}
of $\widetilde{f}_1^W$. Near $r=0$ a fundamental system of solutions $(\phi_1,\phi_2)$ is given by
\begin{align*}
    &\phi_1(r):= r^0 \varphi_1(r), &\phi_2(r):= r^{2-d}\varphi_2(r)+ C \log(r)\phi_1(r)
\end{align*}
with analytic $\varphi_1, \varphi_2$ near $r=0$ with $\varphi_1(0)=1=\varphi_2(0)$. Now $0 \neq \widetilde{f}_1^W \in H^{k}\left((0,1);(\cdot)^\frac{d-1}{2}\right)$ dictates $\widetilde{f}_1^W$ to be a nonzero multiple of $\phi_1$ since $\phi_2$ is not in $H^k\left((0,1);(\cdot)^\frac{d-1}{2}\right)$ for all $d \geq 5$ and $k \in \N$.

Near $r=1$ a fundamental system of solutions $(\phi_1,\phi_2)$ is given by
\begin{align*}
    &\phi_1(r)= (1-r)^0 \varphi_1(r), &\phi_2(r)=(1-r)^{-\lambda-\frac{d-3}{2}}\varphi_2(r)+ C \log(1-r)\phi_1(r)
\end{align*}
again with analytic $\varphi_1, \varphi_2$ near $r=1$ with $\varphi_1(1)=1=\varphi_2(1)$.
We note that $\operatorname{Re}(-\lambda - \frac{d-3}{2}) <0$ by the assumption on $\lambda$. Hence $\phi_2'$ is not square-integrable. Thus the solution must be a nonzero multiple of $\phi_1$ near $r=1$.

In summary we have shown $\widetilde{f}_1^W \neq 0$ is smooth on all of $[0,1]$ and solves
\begin{align*}
    ((\cdot)^2-1) (\widetilde{f}_{1}^W)''+\frac{2(\lambda+d-1)(\cdot)^2+1-d}{(\cdot)} (\widetilde{f}_{1}^W)'+\left((\lambda+d-1)(\lambda+d-2)+V\right)\widetilde{f}_{1}^W=0
\end{align*}
in the classical sense and its behaviour near both endpoints is given by $\widetilde{f}_1^W(r) \simeq 1$.
\end{proof}
\subsection{Sturm-Liouville}
The previous lemma motivates us to investigate the existence of smooth solutions to
\begin{align}\label{smoothsol}
    ((\cdot)^2-1) f''+\frac{2(\lambda+d-1)(\cdot)^2+1-d}{(\cdot)} f'+\left((\lambda+d-1)(\lambda+d-2)+V\right) f=0.
\end{align}
We will use a standard argument from Sturm-Liouville theory. Towards this we reduce the equation into standard Liouville form. That is, consider
\begin{align*}
   f''+p f'+q f=0,
\end{align*}
with $p(r):=-\frac{2(\lambda+d-1)r^2+1-d}{r(1-r^2)}, q(r):= -\frac{(\lambda+d-1)(\lambda+d-2)+V(r)}{1-r^2}$. With $f(r)=h(r)g(r)$ where $h(r):=e^{-\frac{1}{2}P(r)}$, $P'(r)=p(r)$ we get the equation
\begin{align}\label{reducedode}
    -(1-r^2)^2g''(r)+ \left[\frac{\frac{(d-3)(d-1)}{4}-r^2(\frac{3d-9}{2})}{r^2}+(1-r^2)V(r)\right]g(r)= -(\lambda+1)(\lambda+d-4)g(r).
\end{align}
Let us also record that (up to a multiplicative constant) $h(r)=r^\frac{1-d}{2}(1-r^2)^\frac{1-d-2\lambda}{4}$.
\begin{remark}
    Note that the two solutions to Eq.~\eqref{smoothsol}, $f_1(r):=\frac{1}{d-4+r^2}$ for $\lambda=4-d$ and $f_2(r):=\frac{(1-r^2)^\frac{5-d}{2}}{d-4+r^2}$ for $\lambda=-1$, induced by the symmetries of time-translation and time-translation composed with conformal involution respectively, are transformed to the same solution $g_0(r):=\frac{r^{\frac{d-1}{2}}(1-r^2)^\frac{7-d}{4}}{d-4+r^2}.$ (Note however that $f_2$ is not smooth for $d > 5$; for $d=5$ the functions $f_1$ and $f_2$ coincide.)
\end{remark}
From \Cref{inducelemma} we know that for $\operatorname{Re}\lambda > -1$ the behaviour of a solution $\phi$ to Eq.~\eqref{smoothsol} induced by an eigenfunction of the original problem is determined by $\phi(r) \simeq 1$ near $0$ and $\phi(r) \simeq 1$ near $1$. A solution $\phi$ is transformed into a solution $\phi_h$ of Eq.~\eqref{reducedode} that behaves like $\phi_h(r) \simeq r^\frac{d-1}{2}$ near $0$ and $\phi_h(r) \simeq (1-r)^\frac{d-1+2\lambda}{4}$ near $1$. Hence $\phi_h$ lies in the space $L^2((0,1);\frac{1}{1-(\cdot)^2})$ for $\operatorname{Re}\lambda > -1$ and vice-versa which yields a one-to-one correspondence. This inspires us to switch to the $L^2$-setting, since here the operator in question will be self-adjoint.

So we are lead to consider the singular Sturm-Liouville operator from Eq.~\eqref{reducedode}.
\begin{definition}
We denote by $\tau$ the formal differential operator
\begin{align*}
    \tau g:=-\frac{1}{w} \left[g''+ q g\right],
\end{align*}
where $w(r)=\frac{1}{(1-r^2)^2}$, $q(r):=-\left[\frac{\frac{(d-3)(d-1)}{4}-r^2(\frac{3d-9}{2})}{r^2(1-r^2)^2}+\frac{V(r)}{1-r^2}\right]$. Define the operator $T$ by
\begin{align*}
    &T:D(T) \subset L^2((0,1);\sqrt{w}) \rightarrow L^2((0,1);\sqrt{w}), &Tf:= \tau f,
\end{align*} 
with domain $D(T):=\{u \in L^2((0,1);\sqrt{w}) \mid u \in AC_{\operatorname{loc}}(0,1), u' \in AC_{\operatorname{loc}}(0,1), T u \in L^2((0,1);\sqrt{w})\}$. 
\end{definition}
\begin{remark}
    As is standard we require minimal assumptions on the domain to make sense of the differentiation occurring in $\tau$. It is readily seen that $D(T)$ is general enough to include the transformed solutions $\phi_h$ if they exist.
\end{remark}
Let $\mu:= -(\lambda+1)(\lambda+d-4)$.
Frobenius analysis (or our beforehand discussion) reveals that the Frobenius indices for solutions to $\tau g=\mu g$ are
$\{\frac{d-1}{2},-\frac{d-3}{2}\}$ at $0$ and $\{\frac{2 \pm \sqrt{(d-5)^2-4 \mu}}{4}\}$ at $1$. In particular, for $\mu = 4-d$ (which corresponds to $\lambda=0$) we have $\{\frac{d-1}{2},-\frac{d-3}{2}\}$ at $0$ and $\{\frac{d-1}{4},-\frac{d-1}{4}\}$ at $1$. At each endpoint one behaviour is not $L^2((0,1);\sqrt{w})$-integrable, which reveals that $T$ is in the limit-point case at both endpoints by \textit{Weyl's alternative} \cite[p.~82, Theorem 5.6]{Weylbook}. According to Weyl theory of self-adjoint extensions \cite[p.~83, Theorem 5.7]{Weylbook}, being limit-point at both ends implies that $T$ defined on $D(T)$ coincides with the closure of the restriction $T\mathord\restriction \{u \ \text{has compact support in} \ (0,1)\}$ which is a self-adjoint operator with compact resolvent. Hence its spectrum consists only of a countable number of real, simple eigenvalues.

Ruling out eigenvalues $\mu<0$ in the self-adjoint problem implies that there exist no eigenvalues $\operatorname{Re}\lambda>-1$ in the original problem by the following argument. Let $\lambda= a+bi$. Suppose $a>-1$. We have $\operatorname{Re}\mu=-a^2+b^2-(d-3)a-d+4$ and $\operatorname{Im}\mu=-b(2a+d-3).$ Eigenvalues in the self-adjoint problem are real, hence $\operatorname{Im}\mu=0$ and the assumption on $a$ implies $b=0$. Hence $\operatorname{Re}\mu=-a^2-(d-3)a-d+4 < 0$ by the assumption on $a$.

We will rule out eigenvalues $\mu < 0$ by using the solution $g_0(r)=\frac{r^{\frac{d-1}{2}}(1-r^2)^\frac{7-d}{4}}{d-4+r^2}$ from above which solves $\tau g_0=0$ and has no zero on $(0,1)$. From the Frobenius analysis we know that there exist unique solutions to 
\begin{align*}
    \tau g_\mu= \mu g_\mu, \hspace{1cm} g_\mu(0)=0, \hspace{1cm} g_\mu^{\left(\frac{d-1}{2}\right)}(0)=1.
\end{align*}
The Sturm comparison theorem then reads as follows.
\begin{lemma}
    Let $\mu \in \R_0^-$. Then the unique solution $g_\mu$ to
    \begin{align*}
        \tau g_\mu= \mu g_\mu, \hspace{1cm} g_\mu(0)=0, \hspace{1cm} g_\mu^{\left(\frac{d-1}{2}\right)}(0)=1,
    \end{align*}
    has no zero on $(0,1)$.
\end{lemma}
\begin{proof}
    We follow the proof in \cite[p.~34, Theorem 2.6.3]{SLTheorybook}. The case $\mu =0$ is handled by the explicit description of $g_0$. For $\mu <0$ we argue by contradiction. Assume $y_0$ to be the first zero of $g_\mu$, which would exist since $\tau$ is non-oscillatory at $0$. The \textit{Picone identity}
\begin{align*}
    \left[\frac{g_\mu}{g_0}(g_\mu' g_0-g_\mu g_0')\right]'= \frac{(g_\mu' g_0- g_\mu g_0')^2}{(g_0)^2} - w \mu (g_\mu)^2
\end{align*}
is computed straightforwardly. Integrating from $0$ to $y_0$ yields
\begin{align*}
    0=\int_0^{y_0} \frac{(g_\mu' g_0- g_\mu g_0')^2}{(g_0)^2} -\mu\int_0^{y_0} w (g_\mu)^2
\end{align*}
where we used $g_\mu(0)=g_0(0)=0=g_\mu(y_0)$. Since $\mu <0$ and $w>0$ on $(0,y_0)$ we conclude $g_\mu \equiv 0$ on $[0,y_0]$. This contradicts $g_\mu^{\left(\frac{d-1}{2}\right)}(0)=1$, which completes the proof.
\end{proof}
An adjusted version of the Sturm oscillation theorem then finishes the full argument.
\begin{lemma}\label{sturmoscillation}
    The self-adjoint operator $T$ has no eigenvalue $\mu <0$.
\end{lemma}
\begin{proof}
    We invoke \cite[Theorem 1.2]{teschlzeroes}\footnote{Note that there is a small typo in Theorem 1.2 in \cite{teschlzeroes}: $N(c)$ is supposed to denote the number of zeroes of $u_2$ minus the number of zeroes of $u_1$, not the other way round.} which states that there are as many eigenvalues of $T$ in the interval $[\nu,0)$ as there are zeroes of $g_0$ minus the zeroes of $g_{\nu}$, i.e.~$0-0=0$. Since this is true for any $\nu < 0$ there are no eigenvalues $\mu <0$.
\end{proof}
\subsection{The linearised evolution is uniformly exponentially stable} To summarise this section we have now shown that the half-plane $\{\lambda \in \C \mid \operatorname{Re}\lambda > -1\}$ lies in the resolvent set $\rho(\mathbf{L})$ of the generator $\mathbf{L}$ of the linearised evolution. This together with \Cref{gammar} implies the exponential stability of the linearised evolution in every odd dimension greater than or equal to $5$ by an application of the Gearhart-Prüss-Greiner Theorem, see e.g.~\cite[p.~302, Theorem 1.11]{engnag}.
\begin{theorem}\label{linearisedsemigroup}
    The operator $\mathbf{L}: \mathcal{D}(\mathbf{L}) \subset \hilspacewrad{k} \rightarrow \hilspacewrad{k}$ generates the strongly continuous semigroup $\mathbf{S}: [0,\infty) \rightarrow \mathcal{L}(\hilspacewrad{k})$. Additionally, there exists $\delta >0$ such that $\mathbf{S}$ satisfies the growth estimate
    \begin{align*}
        \norm{\mathbf{S}(s)\mathbf{f}}{\hilspacew{k}} \lesssim e^{-\delta s}\norm{\mathbf{f}}{\hilspacew{k}}
    \end{align*}
    for all $\mathbf{f} \in \hilspacewrad{k}$ and $s \geq 0$. 
\end{theorem}
\begin{proof}
    The generation of the semigroup was the content of \Cref{linearisedevoexist}. To prove the estimate, denote by $\mathbb{H}_0:= \{\lambda \in \C \mid \operatorname{Re}\lambda > 0 \}$ the open half-plane right of $0$ and recall the set $\Gamma_R= \{z \in \mathbb{C} \mid \operatorname{Re}(z) > \frac{\omega_{d,k,\varepsilon}-(d-2)}{2}, \abs{z}> R\}$ from \Cref{gammar}. To apply \cite[p.~302, Theorem 1.11]{engnag} we need the uniform bound
    \begin{align*}
        \sup_{\lambda \in \mathbb{H}_0} \norm{\mathbf{R}(\lambda,\mathbf{L})}{\hilspacew{k}} < \infty.
    \end{align*}
    Divide the closed half-plane $\overline{\mathbb{H}_0}=\{\lambda \in \C \mid \operatorname{Re}\lambda \geq 0\}$ into $\overline{\mathbb{H}_0} \cap \Gamma_R$ and $\overline{\mathbb{H}_0} \cap (\Gamma_R)^C$. The estimate on the first set follows by \Cref{gammar}. The estimate on the second set follows by the fact that the resolvent is analytic with respect to $\lambda$, hence it is bounded on the compact set $\overline{\mathbb{H}_0} \cap (\Gamma_R)^C$. Invoking the Gearhart-Prüss-Greiner Theorem yields the desired uniform exponential stability.
\end{proof}

%% file: NonlinearEvo.tex
\section{Nonlinear evolution}
Since we will make use of the algebra properties of the appearing Sobolev spaces, throughout this section let $d,k \in \N$ with
    \begin{align*}
        d \geq 5 \ \text{odd} \ \ \text{and} \ \ k > d/2.
    \end{align*}
    We have now assembled all prerequisites to consider the full equation \eqref{fulleq}
\begin{align*}
    0=\left(\widehat{\square}+V(y)\right) v(s,y)+ N(v(s,\cdot))(y),
\end{align*}
where
\begin{align*}
        &\widehat{\square}v(s,y)=\left[\partial_s^2+2 \partial_s \Lambda_y+ (y^k y^\ell- \delta^{k\ell}) \partial_k \partial_\ell + \frac{2d-3-3\abs{y}^2}{1-\abs{y}^2}\partial_s +4\Lambda_y+2\right] v(s,y),\\
        &V(y)v(s,y)=-\frac{8(d-3)(d-4)}{(d-4+\abs{y}^2)^2}v(s,y),\\
        &N(v(s,\cdot))(y)=\frac{d-3}{2}\frac{\sin\left(2 \abs{y}\left(u^*_\Psi(y)+v(s,y)\right)\right)-2\abs{y}v(s,y)-\sin\left(2 \abs{y}u^*_\Psi(y)\right)}{\abs{y}^3}-V(y)v(s,y),
\end{align*}
with $u^*_\Psi(y)= \frac{2}{\abs{y}}\arctan{\left(\frac{\abs{y}}{\sqrt{d-4}}\right)}$ being the solution to Eq.~\eqref{geowaveeq}, which is static in FHSC.
To summarise the discussion in the introduction and the results of the previous sections, we record that, in the semigroup picture, we formally have
\begin{align}\label{semigroupode}
    \pt_s \Phi_v(s)= \mathbf{L}\Phi_v(s)+ \mathbf{N}(\Phi_v(s))+\begin{bmatrix}
        0\\
       \left[(\overline{\square} \phi + \mathcal{V}\phi+\mathcal{N}(\phi))\circ \Psi\right](s,\cdot)
    \end{bmatrix}
\end{align}
for the variable pair introduced in Eq.~\eqref{variablepair}
\begin{align*}
    \Phi_v(s)(y):=\begin{bmatrix}
        v(s,y)\\
        \left(\partial_s+ \Lambda_y+(d-2)+(d-3)\frac{\abs{y}^2}{1-\abs{y}^2}\right)v(s,y)
    \end{bmatrix}
\end{align*}
with $v= \phi \circ \Psi$ and the nonlinear remainder operator densely defined for $\mathbf{f}=(f_1,f_2) \in \cinftyradw$ by 
\begin{align*}
    \mathbf{N}(\mathbf{f}):=\begin{bmatrix}
        0\\ N(f_1)
    \end{bmatrix}.
\end{align*}

Our aim in this section is to incorporate the nonlinear term into the functional-analytic framework constructed in the previous sections. 
\subsection{Properties of the nonlinearity}
To this end, we first investigate the specific form of $N$ and prove a local Lipschitz bound, which will be the main ingredient in the subsequent contraction argument. From the discussion in the introduction, it follows that the nonlinear remainder term is a second-order term of a Taylor expansion independent of time.
\begin{lemma}\label{buildingblock}
    Let $(y,z) \in \R^d \times \C$. Let $\widetilde{F}(y,z):=\frac{d-3}{2}\frac{\sin(2\abs{y}z)-2\abs{y}z}{\abs{y}^3}$. Denote $\widetilde{F}'(y,z):=\pt_z \widetilde{F}(y,z)$. Then we have
    \begin{align*}
        N(f)(y)= \widetilde{F}\left(y,u_\Psi^*(y)+f(y)\right)-\widetilde{F}\left(y,u_\Psi^*(y)\right)-\widetilde{F}'\left(y,u_\Psi^*(y)\right)f(y)
    \end{align*}
    for $f \in C^\infty_{\operatorname{rad}}(\overline{\B^d};W)$. Furthermore, $N(f) \in C^\infty_{\operatorname{rad}}(\overline{\B^d};W)$ for all $f \in C^\infty_{\operatorname{rad}}(\overline{\B^d};W).$
\end{lemma}
\begin{proof}
    The proof of the identity is a straightforward comparison of both sides. The additional statement follows from the form of $N$ introduced in \cite[Lemma 4.1]{MR4778061}.
\end{proof}
The point is that the quadratic term $\mathbf{N}$ enjoys a local Lipschitz bound. 
\begin{lemma}\label{nonlineraityextends}
    The map $\mathbf{N}: \cinftyradw \rightarrow \cinftyradw$ has a unique extension to a map $\mathbf{N}: \hilspacewrad{k} \rightarrow \hilspacewrad{k+1} \subset \hilspacewrad{k}$ with the property that for any $M >0$ the estimates
    \begin{align}
        \norm{\mathbf{N}(\mathbf{f})}{\hilspacew{k}} &\leq \norm{\mathbf{N}(\mathbf{f})}{\hilspacew{k+1}} \lesssim \norm{\mathbf{f}}{\hilspacew{k}}^2,\\
        \norm{\mathbf{N}(\mathbf{f})-\mathbf{N}(\mathbf{g})}{\hilspacew{k}} &\lesssim  \norm{\mathbf{f}-\mathbf{g}}{\hilspacew{k}} \left(\norm{\mathbf{f}}{\hilspacew{k}}+\norm{\mathbf{g}}{\hilspacew{k}}\right)
    \end{align}
    hold for every $\mathbf{f},\mathbf{g} \in \hilspacewrad{k}$ with $\norm{\mathbf{f}}{\hilspacew{k}},\norm{\mathbf{g}}{\hilspacew{k}} \leq M$. Furthermore, $\mathbf{N}$ is continuously differentiable in the Fr\'echet sense.
\end{lemma}
\begin{proof}
    Note that the algebra property of $\sobspace{k}$ implies
    \begin{align}\label{AP}
        \norm{ab}{\sobspacew{k}} \lesssim \norm{a}{\sobspacew{k}}\norm{b}{\sobspace{k}}. \tag{AP}
    \end{align}
    for all $a \in \sobspacew{k}$, $b \in \sobspace{k}$.
    
    Towards the first inequality we deduce by \Cref{buildingblock}, the form of $N$ introduced in \cite[Lemma 4.1]{MR4778061} and \eqref{AP} that
    \begin{align*}
        \norm{\mathbf{N}(\mathbf{f})}{\hilspacew{k+1}} &= \norm{N(f_1)}{\sobspacew{k}}\\
        &\lesssim \norm{f_1}{\sobspacew{k}}\norm{f_1 \int_0^1 (1-z) \widetilde{F}''\left(\cdot,u_\Psi^*(\cdot)+z f_1(\cdot)\right) \, \mathrm{d}z}{\sobspace{k}}\\
        &\lesssim \norm{f_1}{\sobspacew{k}}^2 \int_0^1 (1-z)\norm{\widetilde{F}''\left(\cdot,u_\Psi^*(\cdot)+z f_1(\cdot)\right)}{\sobspace{k}} \, \mathrm{d}z
    \end{align*}
    for all $\mathbf{f} \in \cinftyradw$, where $\widetilde{F}''(y,x):=\pt_x^2 \widetilde{F}(y,x)= -2(d-3)\frac{\sin\left(2\abs{y}x\right)}{\abs{y}}$. Note that $\widetilde{F}''$ is smooth in $y$ and $\widetilde{F}''(y,0)=0$. Using \textit{Schauder's lemma} \cite[p.~217, Theorem 6.2.2]{Moserbook} we can estimate
    \begin{align*}
        \norm{\widetilde{F}''\left(\cdot,u_\Psi^*(\cdot)+z f_1(\cdot)\right)}{\sobspace{k}} 
        &\lesssim \varphi\left(\norm{u_\Psi^*+z f_1}{\sobspace{k}}\right)
    \end{align*}
    for some smooth $\varphi$.
    In total, noticing $\norm{f_1}{\sobspace{k}} \lesssim \norm{f_1}{\sobspacew{k}} \lesssim \norm{\mathbf{f}}{\hilspacew{k}}$, and using that $\varphi$ is bounded on compacta we get
    \begin{align*}
        \norm{\mathbf{N}(\mathbf{f})}{\hilspacew{k+1}} \lesssim \norm{\mathbf{f}}{\hilspacew{k}}^2
    \end{align*}
    for all $\norm{\mathbf{f}}{\hilspacew{k}} \leq M$. This proves the first claim for all $\mathbf{f} \in C^\infty_{\operatorname{rad}}(\overline{\B^d};W)^2$ with $\norm{\mathbf{f}}{\hilspacew{k}} \leq M$.
    
    Similarly towards the second claim we use \Cref{buildingblock}, \cite[Lemma 4.1]{MR4778061} and \eqref{AP} to estimate
    \begin{align*}
        \norm{\mathbf{N}(\mathbf{f})-\mathbf{N}(\mathbf{g})}{\hilspacew{k}} &\lesssim \norm{N(f_1)-N(g_1)}{\sobspace{k}}\\
        &\lesssim \norm{f_1-g_1}{\sobspacew{k}}\left(\norm{f_1}{\sobspacew{k}}+\norm{g_1}{\sobspacew{k}}\right)\\
        &\quad \times \int_0^1 \int_0^1 \norm{\widetilde{F}''\left(\cdot,u_\Psi^*(\cdot)+s(t f_1(\cdot)+(1-t)g_1(\cdot))\right)}{\sobspace{k}} \, \mathrm{d}t \mathrm{d}s
    \end{align*}
    for all $\mathbf{f}, \mathbf{g} \in C^\infty_{\operatorname{rad}}(\overline{\B^d};W)^2$.
    As for the first claim we conclude
    \begin{align*}
        \int_0^1 \int_0^1 \norm{\widetilde{F}''\left(\cdot,u_\Psi^*(\cdot)+s(t f_1(\cdot)+(1-t)g_1(\cdot))\right)}{\sobspace{k}} \, \mathrm{d}t \mathrm{d}s \lesssim 1
    \end{align*}
    for all $\norm{\mathbf{f}}{\hilspacew{k}}, \norm{\mathbf{g}}{\hilspacew{k}} \leq M$. This implies the second estimate for all $\mathbf{f},\mathbf{g} \in \cinftyradw$ with $\norm{\mathbf{f}}{\hilspacew{k}}, \norm{\mathbf{g}}{\hilspacew{k}} \leq M$.
    The estimates allow to uniquely continuously extend $\mathbf{N}$ to a map $\mathbf{N}:\hilspacewrad{k} \rightarrow \hilspacewrad{k}$ with image in $\hilspacewrad{k+1}$ satisfying the claimed estimates.
    
    For the additional statement we first observe that $\mathbf{N}$ is Fr\'echet differentiable, i.e.~that the prospective Fr\'echet derivative of $\mathbf{N}$ at $\mathbf{f}$, given by the multiplication operator
    \begin{align*}
        \mathbf{h \mapsto }D\mathbf{N}(\mathbf{f})\mathbf{h} =\begin{bmatrix}
            0 \\ N'(f_1)h_1
        \end{bmatrix},
    \end{align*}
    where $N'(f_1):= 2f_1 \int_0^1 (1-z)\widetilde{F}''(\cdot,u_\Psi^*(\cdot)+z f_1(\cdot)) \, \mathrm{d}z+ f_1^2 \int_0^1 (1-z)\widetilde{F}'''(\cdot,u_\Psi^*(\cdot)+z f_1(\cdot))z \, \mathrm{d}z$ with $\widetilde{F}'''(y,z):=\pt_z \widetilde{F}''(y,z)$, is bounded as an operator $\hilspacewrad{k} \rightarrow \hilspacewrad{k}$ for every $\mathbf{f} \in \hilspacewrad{k}$. This follows readily since $N'(f_1) \in \sobspacew{k}$ by a similar argument to the one in the proof above for $N(f_1) \in \sobspacew{k}$. 
    
    Lastly we prove continuity of the map $\hilspacewrad{k} \rightarrow \mathcal{L}(\hilspacewrad{k})$, $\mathbf{f} \mapsto D\mathbf{N}(\mathbf{f})$. We note that by \eqref{AP} we have
    \begin{align*}
        \norm{D\mathbf{N}(\mathbf{f})-D\mathbf{N}(\mathbf{g})}{\mathcal{L}(\hilspacewrad{k})}&=\sup_{\norm{\mathbf{h}}{\hilspacew{k}}=1}\norm{(D\mathbf{N}(\mathbf{f})-D\mathbf{N}(\mathbf{g}))\mathbf{h}}{\hilspacew{k}}\\
        &=\sup_{\norm{\mathbf{h}}{\hilspacew{k}}=1}\norm{(N'(f_1)-N'(g_1))h_1}{\sobspacew{k-1}}\\
        &\leq\sup_{\norm{\mathbf{h}}{\hilspacew{k}}=1} \norm{(N'(f_1)-N'(g_1))h_1}{\sobspacew{k}} \\
        &\lesssim \norm{N'(f_1)-N'(g_1)}{\sobspace{k}}
    \end{align*}
    and hence it is sufficient to prove continuity of $N': \sobspacewrad{k} \rightarrow \sobspacerad{k}$. The fact that the occurring spaces are Banach algebras allows us to reduce this to continuity of every appearing term. Using the fundamental theorem of calculus for the maps $f_1 \mapsto {\widetilde{F}''}{}^{(}{}'{}^{)}\left(\cdot,u_\Psi^*(\cdot)+z f_1(\cdot) \right)$
    respectively, shows that
    \begin{align*}
        &{\widetilde{F}''}{}^{(}{}'{}^{)}\left(\cdot,u_\Psi^*(\cdot)+z f_1(\cdot) \right)-{\widetilde{F}''}{}^{(}{}'{}^{)}\left(\cdot,u_\Psi^*(\cdot)+z g_1(\cdot) \right)\\
        &= \int_0^1 {\widetilde{F}'''}{}^{(}{}'{}^{)}\left(\cdot,u_\Psi^*(\cdot)+z (t f_1(\cdot)+(1-t)g_1(\cdot)) \right)z(f_1(\cdot)-g_1(\cdot)) \, \mathrm{d}t.
    \end{align*}
    Smoothness of $x \mapsto \widetilde{F}''(y,x)$ yields that these terms are bounded by $\varphi(\norm{f_1-g_1}{\sobspacew{k}})$ for some smooth $\varphi$, again by Schauder's Lemma. It follows that $N'$ is locally Lipschitz continuous and in particular continuous. This completes the proof.
\end{proof}
\subsection{Notion of solution}
Since we just gained control over the nonlinearity $\mathbf{N}$ we can meaningfully apply Duhamel's principle to adopt a notion of solution from \cite[p.~184, Definition 1.1]{pazybook}.
\begin{definition}\label{defmildsol}
    Recall
    \begin{align*}
        &\mathbf{S}: [0,\infty) \rightarrow \mathcal{L}\left(\hilspacewrad{k}\right) \ \text{the semigroup from \Cref{linearisedsemigroup}},\\
        &\mathbf{L}: \mathcal{D}(\mathbf{L})\subset \hilspacewrad{k} \rightarrow \hilspacewrad{k} \ \text{its generator from \Cref{linearisedgeneratordef}},\\
        &\mathbf{N}: \hilspacewrad{k} \rightarrow \hilspacewrad{k} \ \text{the nonlinearity from \Cref{nonlineraityextends}}.
    \end{align*}
    Let $\mathbf{f} \in \hilspacewrad{k}.$\\
    A map $\mathbf{v} \in C\left([0,\infty),\hilspacewrad{k}\right) \cap C^1 \left((0,\infty),\hilspacewrad{k}\right)$ which satisfies $\mathbf{v}(s) \in \mathcal{D}(\mathbf{L})$ and
    \begin{align*}
    \begin{cases}
        &\partial_s \mathbf{v}(s)= \mathbf{L}\mathbf{v}(s)+ \mathbf{N}(\mathbf{v}(s)),\\
        &\mathbf{v}(0)=\mathbf{f},
    \end{cases}
    \end{align*}
    for all $s > 0$ is called a \textit{classical solution} to the abstract Cauchy problem.\\
    A map $\mathbf{v} \in C([0,\infty),\hilspacewrad{k})$ which satisfies
    \begin{align}\label{fixedpointequation}
        \mathbf{v}(s)= \mathbf{S}(s)\mathbf{f}+ \int_0^s \mathbf{S}(s-s')\mathbf{N}(\mathbf{v}(s')) \, \mathrm{d}s'
    \end{align}
    for all $s > 0$, is called a \textit{mild solution} to the abstract Cauchy problem.
\end{definition}
We now turn to forward-in-time stable solutions to the abstract Cauchy problem (ACP). To this end, we incorporate the decay of the linearised wave flow into an appropriate function space.
\begin{definition}
    Let $\delta >0$ be the growth constant from \Cref{linearisedsemigroup}. Define the Banach space $(\fixptspace{k}, \norm{\cdot}{\fixptspace{k}})$ by
    \begin{align*}
        &\fixptspace{k}:=\left\{\mathbf{v} \in C\left([0,\infty),\hilspacewrad{k}\right) \mid \norm{\mathbf{v}(s)}{\hilspacew{k}} \lesssim e^{-\delta s} \ \text{for all} \ s \geq 0\right\},\\
        &\norm{\mathbf{v}}{\fixptspace{k}}:=\sup_{s \in [0,\infty)}\left(e^{\delta s}\norm{\mathbf{v}(s)}{\hilspacew{k}}\right). 
    \end{align*}
\end{definition}
\subsection{Nonlinear asymptotic stability}
We carry out a standard fixed-point argument to show global well-posedness, i.e.~semi-global existence, uniqueness and continuous dependence on the data.
\begin{lemma}\label{fixptarg}
    There exists $0 < \varepsilon_0 < 1$ such that for all $0 < \varepsilon < \varepsilon_0$, there exists $C_{\varepsilon}> 1$ such that for all $\mathbf{f} \in \hilspacewrad{k}$ with $\norm{\mathbf{f}}{\hilspacew{k}} \leq \frac{\varepsilon}{C_\varepsilon}$ there is a unique $\mathbf{v}_{\mathbf{f}} \in \fixptspace{k}$ which satisfies
    \begin{align*}
        \mathbf{v}_{\mathbf{f}}(s)=\mathbf{S}(s)\mathbf{f}+ \int_0^s \mathbf{S}(s-s') \mathbf{N}(\mathbf{v}_{\mathbf{f}}(s')) \, \mathrm{d}s'
    \end{align*}
    for all $s \geq 0$. Furthermore,  $\mathbf{v}_\mathbf{f}$ satisfies $\norm{\mathbf{v}_{\mathbf{f}}}{\fixptspace{k}} \leq \varepsilon$ and we have
    \begin{align*}
        \norm{\mathbf{v}_{\mathbf{f}}-\mathbf{v}_{\mathbf{g}}}{\fixptspace{k}} \lesssim \norm{\mathbf{f}-\mathbf{g}}{\hilspacew{k}}
    \end{align*}
    for all $\mathbf{f},\mathbf{g} \in \hilspacewrad{k}$ with $\norm{\mathbf{f}}{\hilspacew{k}},\norm{\mathbf{g}}{\hilspacew{k}} \leq \frac{\varepsilon}{C_\varepsilon}$.
\end{lemma}
\begin{proof}
    Let $0<\varepsilon_0$ be arbitrary but fixed and $0 < \varepsilon < \varepsilon_0$. Set
    \begin{align*}
        B_{\varepsilon}:=\{\mathbf{v} \in \fixptspace{k} \mid \norm{\mathbf{v}}{\fixptspace{k}} \leq \varepsilon \}
    \end{align*}
    the closed ball of radius $\varepsilon$ in $\fixptspace{k}$. For $\mathbf{f} \in \hilspacewrad{k}$ and $\mathbf{v} \in \fixptspace{k}$ define
    \begin{align*}
        \mathbf{K}_{\mathbf{f}}(\mathbf{v})(s):= \mathbf{S}(s)\mathbf{f}+ \int_0^s \mathbf{S}(s-s')\mathbf{N}(\mathbf{v}(s')) \, \mathrm{d}s'.
    \end{align*}
     The estimates on the nonlinearity from \Cref{nonlineraityextends} imply $\norm{\mathbf{N}(\mathbf{v}(s'))}{\hilspacew{k}} \lesssim e^{-2\delta s'}\varepsilon^2$ and \\
     $\norm{\mathbf{N}(\mathbf{u}(s'))-\mathbf{N}(\mathbf{v}(s'))}{\hilspacew{k}} \lesssim e^{-\delta s'}\varepsilon \norm{\mathbf{u}(s')-\mathbf{v}(s')}{\hilspacew{k}}$ for all $s' \geq 0$ and $\mathbf{u},\mathbf{v} \in B_{\varepsilon}$. This yields 
    \begin{align*}
        \norm{\mathbf{K}_{\mathbf{f}}(\mathbf{v})(s)}{\hilspacew{k}} &\lesssim e^{-\delta s}\norm{\mathbf{f}}{\hilspacew{k}}+ e^{-\delta s} \varepsilon^2,\\
        \norm{\mathbf{K}_{\mathbf{f}}(\mathbf{u})(s)-\mathbf{K}_{\mathbf{f}}(\mathbf{v})(s)}{\hilspacew{k}} &\lesssim e^{-\delta s} \varepsilon\int_0^s \norm{\mathbf{u}(s')-\mathbf{v}(s')}{\hilspacew{k}}\lesssim e^{-\delta s}\varepsilon \norm{\mathbf{u}-\mathbf{v}}{\fixptspace{k}}
    \end{align*}
    for all $\mathbf{u}, \mathbf{v} \in B_{\varepsilon}$. Hence $\varepsilon_0$ can be chosen such that for all $0< \varepsilon < \varepsilon_0$ there exists $C_\varepsilon> 1$ such that the map $B_{\varepsilon} \rightarrow B_{\varepsilon}, \mathbf{v} \mapsto \mathbf{K}_{\mathbf{f}}(\mathbf{v})$ is well-defined and a contraction for all $\mathbf{f} \in \fixptspace{k}$ with $\norm{\mathbf{f}}{\fixptspace{k}} \leq \frac{\varepsilon}{C_\varepsilon}$. Banach's fixed-point theorem reveals existence and uniqueness of $\mathbf{v}_{\mathbf{f}} \in B_\varepsilon$ such that $\mathbf{K}_{\mathbf{f}}(\mathbf{v}_{\mathbf{f}})=\mathbf{v}_{\mathbf{f}}$. Uniqueness in all of $\fixptspace{k}$ we conclude by Gr\"onwall's inequality using the first inequality in the second estimate above. The additional estimate is deduced by
    \begin{align*}
        \norm{\mathbf{K}_{\mathbf{f}}(\mathbf{v}_\mathbf{f})(s)-\mathbf{K}_{\mathbf{g}}(\mathbf{v}_\mathbf{g})(s)}{\hilspacew{k}} \leq l \norm{\mathbf{v}_{\mathbf{f}}-\mathbf{v}_{\mathbf{g}}}{\fixptspace{k}}+ \norm{\mathbf{f}-\mathbf{g}}{\hilspacew{k}}
    \end{align*}
    for all $s \geq 0$, which follows from triangle inequality with constant $l < 1$ for $\mathbf{K}_\mathbf{f}$ is a contraction. This finishes the proof.
\end{proof}
\begin{remark}
    Note that since $\mathbf{0}$ is the unique fixed point of $\mathbf{K}_{\mathbf{0}}$, the continuous dependence implies the estimate $\norm{\mathbf{v}_\mathbf{f}}{\fixptspace{k}} \lesssim \norm{\mathbf{f}}{\hilspacew{k}}$.
\end{remark}

%% file: proofofmainresults.tex
\section{Proof of main results}
\Cref{fixptarg} is the cornerstone of the paper. In what follows, we refine the result under stronger assumptions. Throughout this section, let $d,k \in \N$ with
    \begin{align*}
        d \geq 5 \ \text{odd} \ \ \text{and} \ \ k > d/2.
    \end{align*}
Our main interest is to connect (ACP) with the classical initial value problem. First, we upgrade the mild solution into a classical solution to (ACP).
\begin{lemma}\label{upgradeacp}
     Suppose $\mathbf{v} \in C([0,\infty), \hilspacewrad{k})$ is a mild solution to (ACP) with initial datum $\mathbf{f} \in \mathcal{D}(\mathbf{L})$. Then $\mathbf{v}$ is the unique classical solution to (ACP).
\end{lemma}
\begin{proof}
    We invoke \cite[p.~187, Theorem 1.5]{pazybook}\footnote{We mention that the proof as written in the book is inaccurate since it is falsely claimed that continuous differentiability of an operator implies Lipschitz continuity. However this can be readily adjusted by using only local Lipschitz continuity with slight adjustments to the proof.} to exploit $\mathbf{N}$ being $C^1$ in the sense of Fr\'{e}chet derivatives, which yields that $\mathbf{v}$ is already the classical solution to (ACP).
\end{proof}
Combining \Cref{fixptarg} and \Cref{upgradeacp} yields \Cref{mainthmfhsc}.

While it is certainly possible to continue assuming less regularity, we restrict ourselves to smooth data. The next step is to observe that under the assumption of smooth data the solution to (ACP) is jointly smooth.
\begin{lemma}\label{jointlysmooth}
     Suppose $\mathbf{v} \in C([0,\infty), \hilspacewrad{k})$ is a mild solution to (ACP) with initial datum $\mathbf{f} \in \cinftyradw \subset \mathcal{D}(\mathbf{L})$. Then $\mathbf{v} \in \cinftyjointlyw$.
\end{lemma}
\begin{proof}
     Let $\mathbf{f} \in \cinftyradw$. Since $\mathbf{v}(s) \in \hilspacewrad{k}$ for all $s \geq 0$, \Cref{nonlineraityextends} implies $\mathbf{N}(\mathbf{v}(s)) \in \hilspacewrad{k+1}$ for all $s \geq 0$. Hence the fixed-point equation~\eqref{fixedpointequation} satisfied by the mild solution $\mathbf{v}$ together with the restriction property \cite[p.~34, Lemma C.1]{Glogicrestriction} of $\mathbf{S}$ and the fact that $\cinftyradw$ is invariant under $\mathbf{S}(s)$ for all $s > 0$ recursively shows that $\mathbf{v}(s) \in \hilspacewrad{k+1}$ for all $s > 0, k \geq \lceil\frac{d}{2}\rceil$. In fact, $\mathbf{v} \in C([0,\infty),\hilspacewrad{k+1})$ by strong continuity on $\hilspacewrad{k+1}$ of all operators appearing on the right-hand side of the fixed-point equation for all $k \geq \lceil\frac{d}{2}\rceil$. By Sobolev embedding we have $\mathbf{v}(s) \in \cinftyradw$ for all $s \geq 0$. To show that the same is true for all $s$-derivatives we note by the previous Lemma that $\mathbf{v}$ is the classical solution and hence satisfies
    \begin{align*}
        \pt_s \mathbf{v}(s)=\mathbf{L}\mathbf{v}(s)+ \mathbf{N}(\mathbf{v}(s))
    \end{align*}
    for all $s > 0$. Since $\cinftyradw$ is invariant under $\mathbf{L}$ and $\mathbf{N}$, we have $\pt_s \mathbf{v}(s) \in \cinftyradw$ for all $s >0$. Hence we can consider the equation pointwise at every $y \in \B^d$. Note first that $\mathbf{L}$ is bounded as a map $\hilspacewrad{l+1} \rightarrow \hilspacewrad{l}$ for every $l \in \N$. Thus
    \begin{align*}
        \mathbf{L}\mathbf{v}:[0,\infty) &\rightarrow \left(C \left(\overline{\B^d};W \right)^2,\norm{\cdot}{L^\infty(\B^d;W)^2}\right),\\
        s &\mapsto \mathbf{L}\mathbf{v}(s)
    \end{align*}
    is continuous as the composition of continuous maps, since we can choose $l \in \N$ large enough such that $\hilspacewrad{l}$ continuously embeds into the latter space. By
    \begin{align*}
        \abs{\mathbf{L}\mathbf{v}(s_1)(y_1)-\mathbf{L}\mathbf{v}(s_2)(y_2)} \leq \abs{\mathbf{L}\mathbf{v}(s_1)(y_1)-\mathbf{L}\mathbf{v}(s_1)(y_2)} + \norm{\mathbf{L}\mathbf{v}(s_1)-\mathbf{L}\mathbf{v}(s_2)}{L^\infty(\B^d;W)^2}
    \end{align*}
    it follows that the map $[0,\infty) \times \overline{\B^d} \rightarrow \R, (s,y) \mapsto \mathbf{L}\mathbf{v}(s)(y)$ is continuous. We also observe that since
    \begin{align*}
        \mathbf{N}(\mathbf{v}(s))(y)= \begin{bmatrix}
            0 \\ v_1(s,y)^2 \int_0^1 (1-z)\widetilde{F}''(y,u_\Psi^*(y)+z v_1(s,y)) \, \mathrm{d}z
        \end{bmatrix},
    \end{align*}
    the map $(s,y) \mapsto \mathbf{N}(\mathbf{v}(s))(y)$ is as jointly smooth as $v_1$ is smooth in $s$. Therefore $s \mapsto \pt_s\mathbf{v}(s,y)$ can be continuously extended to $s=0$. For derivatives of higher order, note that $\pt_s \mathbf{v}(s) \in \cinftyradw \subset D(\mathbf{L})$ implies $\pt_s\mathbf{L}\mathbf{v}(s)= \mathbf{L}\pt_s\mathbf{v}(s)$ and with the same arguments as before we conclude inductively $\pt_s^l \mathbf{v}(s) \in \cinftyradw$ for all $l \in \N, s \geq 0$. Hence $\mathbf{v} \in \cinftyjointlyw$. This completes the proof.
\end{proof}
Next we observe that the first component of the solution to (ACP) solves a classical initial value problem. This is a consequence of the construction of the abstract Cauchy problem and the regularity of the solution we just proved.
\begin{lemma}\label{lemclasshypivp}
    Assume as in \Cref{jointlysmooth}. Then $v_1:= \left[\mathbf{v}\right]_1 \in \cinftyjointlywone$ solves
    \begin{align*}
        \begin{cases}
            &\left(\widehat{\square}+V(y)\right) v_1(s,y)+ N(v(s))(y)=0,\\
            &v_1(0,y)=f_1(y),\\
            &\widehat{\nabla}v_1(0,y)=f_2(y),
        \end{cases}
    \end{align*}
    where $\widehat{\nabla}= \partial_s + \Lambda_y + \frac{d-2-\abs{y}^2}{1-\abs{y}^2}$.
\end{lemma}
\begin{proof}
    Under the given assumptions we learned in the proof of \Cref{jointlysmooth} that we can consider (ACP) pointwise. Note that $\left[\mathbf{L}'(\mathbf{v})(s)\right]_1=0=[\mathbf{N}(\mathbf{v}(s))]_1$, so we get
    \begin{align*}
        \begin{cases}
            &\partial_s v_1(s,y)= \left[\mathbf{L}_0\mathbf{v}(s)\right]_1 (y),\\
            &\partial_s v_2(s,y)= \left[\mathbf{L}_0\mathbf{v}(s)\right]_2 (y) + V(y) v_1(s,y) + N(v_1(s))(y),\\
            &v_1(0,y)= f_1(y),\\
            &v_2(0,y)= f_2(y).
        \end{cases}
    \end{align*}
    with $\left[\mathbf{L}_0\mathbf{v}(s)\right]_1= v_2(s,\cdot)- W^{-1}\Lambda (W v_1(s,\cdot))- (d-2)v_1(s,\cdot)$ and $\left[\mathbf{L}_0\mathbf{v}(s)\right]_2 = -W^{-1}\Lambda(W v_2(s,\cdot))-(d-1)v_2(s,\cdot)+ W^{-1}\Delta(W v_1(s,\cdot))$.
    Solving towards $v_1$ leads to
    \begin{align*}
        \begin{cases}
            &v_2(s,y)= \left(\partial_s + \Lambda_y+ \frac{d-2-\abs{y}^2}{1-\abs{y}^2}\right) v_1(s,y),\\
            &\left(\widehat{\square}+ V(y) \right)v_1(s,y) + N(v_1(s))(y)=0,\\
            &v_1(0,y)= f_1(y),\\
            &\widehat{\nabla} v_1(0,y)= f_2(y).
        \end{cases}
    \end{align*}
    This proves the claim.
\end{proof}
By combining \Cref{jointlysmooth} and \Cref{lemclasshypivp} we obtain the following result.
\begin{lemma}\label{thmclasshypivp}
Let  $\mathbf{f}=(f_1,f_2) \in \cinftyradw$ with $\norm{\mathbf{f}}{\hilspacew{k}}$ small enough. Then the unique classical solution $\boldsymbol{\phi}$ to (ACP) is jointly smooth, namely $\boldsymbol{\phi} \in \cinftyjointlyw$ with $\boldsymbol{\phi}(s)$ being radial for every $s \geq 0$. In particular, $\phi_1 := [\boldsymbol{\phi}]_1$ is the unique solution in $\{\phi \in \cinftyjointlywone \mid \phi(s,\cdot) \ \text{radial for all} \ s \geq 0\}$ to the initial value problem
\begin{align*}
\begin{cases}
    0 =(\widehat{\square}+ V(y))\phi_1(s,y)+N(\phi_1(s,\cdot))(y) \ \ &\text{in} \ (0,\infty)\times\B^d,\\
    \phi_1(0,y) = f_1(y)  &\text{in} \ \B^d,\\
    \widehat{\nabla}\phi_1(0,y) =f_2(y)  &\text{in} \ \B^d.
\end{cases}
\end{align*}
There exists $\delta > 0$ such that
\begin{align*}
    \sup_{s \geq 0} \left(e^{\delta s}\norm{\phi_1(s)}{\sobspacew{k}}+e^{\delta s} \norm{\widehat{\nabla}\phi_1(s)}{\sobspacew{k-1}}\right)
    \lesssim \norm{f_1}{\sobspacew{k}}+ \norm{f_2}{\sobspacew{k-1}}.
\end{align*}
\end{lemma}
This small-data result for the evolution of the perturbation induces a large-data result for the initial value problem given by Eq.~\eqref{geowaveeqsesim} by considering $v= w_\Psi^* + \phi_1$.
\begin{corollary}\label{mainthmcolry}
Let $\mathbf{\mathbf{f}} \in \cinftyradw$ with $\norm{\mathbf{f}}{\hilspacew{k}}$ small enough. Then the initial value problem
\begin{align}\label{ivpcolry}
    \begin{cases}
        0= \widehat{\square}v(s,y)+ F(\Psi(s,y),v(s,y)) \ \ &\text{in} \ (0,\infty)\times\B^d,\\
        v(0,y)= u_\Psi^*(0,y)+ f_1(y) &\text{in} \ \B^d, \tag{IVP}\\
        \widehat{\nabla}v(0,y)=\widehat{\nabla}u_\Psi^*(0,y)+f_2(y) &\text{in} \ \B^d,
    \end{cases}
\end{align}
has a unique solution in the affine space
\begin{align*}
    \left\{w_\Psi^*+ \phi \mid \phi \in \cinftyjointlywone, \phi(s,\cdot) \ \text{radial for all} \ s \geq 0\right\}.
\end{align*}
There exists $\delta > 0$ such that
\begin{align*}
    &\sup_{s \geq 0}e^{\delta s} \left(\norm{w_\Psi^*(s)-v(s)}{\sobspacew{k}}+ \norm{\widehat{\nabla}w_\Psi^*(s)-\widehat{\nabla}v(s)}{\sobspacew{k-1}}\right)\\
    &\hphantom{\sup_{s \geq 0} \left(e^{\delta s}\norm{u_\Psi^*(s)-v(s)}{\sobspacew{k}}+e^{\delta s}\right.}\lesssim \norm{f_1}{\sobspacew{k}}+ \norm{f_2}{\sobspacew{k-1}}.
\end{align*}
\end{corollary}
Before reformulating \Cref{mainthmcolry} into physical coordinates, we upgrade the uniqueness to general smooth functions. This can be read as finite speed of propagation in FHSC. We invite the reader to consult \cref{fsoppic1} and \cref{fsoppic2} in the following argument.
\input{finitespeedofproppic}
\begin{lemma}
    The solution in \Cref{mainthmcolry} is unique in $C^\infty([0,\infty) \times \B^d)$.
\end{lemma}
\begin{proof}
    Fix $(s_0,y_0) \in (0,\infty) \times \B^d$. Consider its causal past up to $s=0$ given by the set
    \begin{align*}
        J^{-}(s_0,y_0):= \left\{(s,y) \in [0,s_0) \times \B^d : \left(\frac{e^s}{1-\abs{y}^2}-\frac{e^{s_0}}{1-\abs{y_0}^2}\right)^2 - \left| \frac{e^s y}{1-\abs{y}^2}-\frac{e^{s_0}y_0}{1-\abs{y_0}^2}\right|^2 \geq 0\right\}.
    \end{align*}
    For every $0 \leq s < s_0$ we consider
    \begin{align*}
        D_s:=\{y \in \B^d: (s,y) \in J^{-}(s_0,y_0)\}.
    \end{align*}
    Clearly we have $J^{-}(s_0,y_0)= \bigcup_{0 \leq s < s_0} \{s\} \times D_s$. Carefully simplifying the condition in $J^{-}(s_0,y_0)$ reveals that every $D_s$ is a ball. Indeed, set $R(s):= 1-e^{s-s_0}$, $P(s):= e^{s-s_0}$. We obtain
    \begin{align*}
        D_s = \B^d_{R(s)}(P(s)y_0).
    \end{align*}
    With the variable pair introduced in Eq.~\eqref{variablepair} we define a local energy
    \begin{align*}
        E_{(s_0,y_0)}(w)(s) &:= e^{2(d-2)s} \int_{D_s} \abs{\nabla (W w_1(s,\cdot))}^2 + \abs{W w_1(s,\cdot)}^2+ \abs{W w_2(s,\cdot)}^2\\
        &= \int_{D_s} \abs{\nabla f_1(s,\cdot)}^2 + \abs{f_1(s,\cdot)}^2+ \abs{f_2(s,\cdot)}^2.
    \end{align*}
    Here $f_1,f_2$ are chosen as in Eq.~\eqref{firstordersystemstar} and the equality follows by Eq.~\eqref{connecttoblowupeq}. To differentiate towards $s$, for every $0 \leq s < s_0$ we introduce the map
    \begin{align*}
        \Phi_s: \B^d &\rightarrow D_s,\\
        x &\mapsto R(s)x + P(s) y_0.
    \end{align*}
    We have $\det(D \Phi_s(x))=R(s)^d$. For brevity we denote $g(s,\cdot):=\abs{\nabla f_1(s,\cdot)}^2 + \abs{f_1(s,\cdot)}^2+ \abs{f_2(s,\cdot)}^2$. We compute
    \begin{align*}
        \partial_s E_{(s_0,y_0)}(w)(s)= \int_{D_s} \partial_s g(s,\cdot) &+ R(s)^{d-1} \partial_s R(s) \int_{\pt \B^d} g(s,\Phi_s(x)) \ \mathrm{d}S(x)\\ 
        &+ R(s)^{d-1}\partial_s P(s) \int_{\pt \B^d} y_0 \cdot x \ g(s,\Phi_s(x)) \ \mathrm{d}S(x).
    \end{align*}
    For the first integrand we use the identities from Eq.~\eqref{firstordersystemstar} to compute (for brevity, set $g \equiv g(s,\cdot), f_1 \equiv f_1(s,\cdot), f_2 \equiv f_2(s,\cdot)$ and identify $y^j$ with the map $y \mapsto y^j$)
    \begin{align*}
        \partial_s g &= 2 \pt_s \partial^k f_1 \pt_k f_1 + 2 \pt_s f_1 f_1 +2 \pt_s f_2 f_2\\
        &= 2 \pt^k(f_2- \Lambda f_1)\pt_k f_1 + 2 f_1(f_2- \Lambda f_1) + 2f_2[-\Lambda f_2-f_2+ \Delta f_1 + \square^* f]\\
        &= 2\partial^k (\partial_k f_1 f_2) -2 \pt^k f_1 \pt_k f_1 - \Lambda(\partial^k f_1 \partial_k f_1) + 2f_1 f_2 - \Lambda (f_1^2) -\Lambda (f_2^2) -2f_2^2+ f_2 \square^* f\\
        &= 2 \pt^k (\pt_k f_1 f_2)- \pt_j [y^j g] + (d-2) \pt^k f_1 \pt_k f_1 + d f_1^2 + (d-2) f_2^2 +2 f_2 \square^* f\\
        &\leq \pt^k (\pt_k f_1 f_2)- \pt_j[y^j g]+ d g + 2f_2 \square^*f.
    \end{align*}
    Integration by parts yields
    \begin{align*}
        \int_{D_s} \pt_s g(s,\cdot) \leq d \int_{D_s} g(s,\cdot) &+ 2\int_{D_s}f_2(s,\cdot) \square^* f(s,\cdot)\\
        &+R(s)^{d-1} \int_{\pt \B^d} 2 x^k \pt_k f_1(s, \Phi_s(x))f_2(s,\Phi_s(x)) \ \mathrm{d}S(x)\\
        &- R(s)^d \int_{\pt \B^d} g(s,\Phi_s(x)) \ \mathrm{d}S(x) \\
        &- R(s)^{d-1} P(s) \int_{\pt \B^d} x \cdot y_0 \ g(s, \Phi_s(x)) \ \mathrm{d}S(x)
    \end{align*}
    Using the identities $\pt_s R(s)= -1+ R(s)$, $\pt_s P(s) = P(s)$, we see that for the derivative of the energy we have some cancellations of the boundary terms and are left with
    \begin{align*}
        \partial_s E_{(s_0,y_0)}(w)(s) \leq d E_{s_0,y_0}(w)(s) + 2\int_{D_s}f_2 \square^* f(s,\cdot) &+ R(s)^{d-1} \int_{\pt \B^d} 2 x^k \pt_k f_1(s, \Phi_s(x))f_2(s,\Phi_s(x))
        \\&- R(s)^{d-1} \int_{\pt \B^d} g(s,\Phi_s(x)) \ \mathrm{d}S(x).
    \end{align*}
    Cauchy-Schwarz and $\abs{x}=1$ yields
    \begin{align*}
        2x^k \partial_k f_1(s,\Phi_s(x))f_2(s,\Phi_s(x)) \leq \abs{\nabla f_1(s, \Phi_s(x)}^2 + f_2(s,\Phi(x))^2 \leq g(s,\Phi(x)).
    \end{align*}
    Finally this shows that
    \begin{align*}
        \partial_s E_{(s_0,y_0)}(w)(s) \lesssim E_{s_0,y_0}(w)(s) + \int_{D_s}\abs{\square^* f(s,\cdot)}^2,
    \end{align*}
    where we also used Cauchy-Schwarz for the mixed term.

    We want to use this inequality for the difference of two solutions $v$ and $v'$ of the initial value problem \eqref{ivpcolry} given in \Cref{mainthmcolry}. We remind ourselves of the identity in Eq.~\eqref{freebweq}, which says that
    \begin{align*}
        \square^*(f_v - f_{v'})= e^{(d-2)s} W(y) \widehat{\square}(v-v'),
    \end{align*}
    for $f_v, f_{v'}$ related to $v, v'$ by Eq.~\eqref{functiontransform}. The fact that $v, v'$ solve \eqref{ivpcolry} means $\widehat{\square}(v(s,\cdot)-v'(s,\cdot))= F(\Psi(s,\cdot), v(s,\cdot))-F(\Psi(s,\cdot), v'(s,\cdot))$. Here
    \begin{align*}
        F(\Psi(s,y), v(s,y))= (d-3) \frac{\sin \left(2\abs{y}v(s,y)\right)-2 \abs{y}v(s,y)}{2\abs{y}^3},
    \end{align*}
    from Eq.~\eqref{geowaveeqsesim}. By observing that the function $(y,z) \mapsto (d-3) \frac{\sin \left(2\abs{y}z\right)-2 \abs{y}z}{2\abs{y}^3}$ is smooth we obtain the bound $F(\Psi(s,\cdot), v(s,\cdot))-F(\Psi(s,\cdot), v'(s,\cdot)) \lesssim v(s,\cdot)- v'(s,\cdot)$. This shows that
    \begin{align*}
        \int_{D_s} \abs{\square^*(f_v - f_{v'})}^2 \lesssim e^{2(d-2)s} \int_{D_s} \abs{W(v(s,\cdot)-v'(s,\cdot))}^2 \leq E_{(s_0,y_0)}(v-v')(s).
    \end{align*}
    In total this means that
    \begin{align*}
        \partial_s E_{(s_0,y_0)}(v-v')(s) \lesssim  E_{(s_0,y_0)}(v-v')(s).
    \end{align*}
    Grönwall's inequality and $E_{(s_0,y_0)}(v-v')(0)=0$ proof uniqueness in the FHSC light cone $J^{-}(s_0,y_0)$. Since $(s_0,y_0) \in (0,\infty) \times \B^d$ was arbitrary, this completes the proof.
\end{proof}
In physical coordinates, we obtain a rescaled version of \Cref{mainthm}.
\begin{corollary}
    There exists $\varepsilon >0$ such that for every $(f_1, f_2)$ defined on $\Sigma_0$ for which $(f_1 \circ \Psi(0,\cdot), f_2 \circ \Psi(0,\cdot)) \in \cinftyradw$ with 
    \begin{align*}
        \sqrt{\norm{f_1 \circ \Psi(0,\cdot)}{H^k\left(\B^d;W\right)}^2+\norm{f_2 \circ \Psi(0, \cdot)}{H^{k-1}\left(\B^d;W\right)}^2} \leq \varepsilon,
    \end{align*} the hyperboloidal initial value problem
    \begin{align*}
        \begin{cases}
            0 = \square_{t,x} \left(\frac{1}{t}w(t,x)\right)+ (d-3)\frac{\sin\left(2\frac{\abs{x}}{t}w(t,x)\right)-2\frac{\abs{x}}{t}w(t,x)}{2\abs{x}^3} \ \ &\text{in} \ D^+(\Sigma_0)\setminus \Sigma_0,\\
            w = w^*+f_1 &\text{in} \ \Sigma_0,\\
            \overline{\overline{\nabla}}w =\overline{\overline{\nabla}}w^*+f_2 &\text{in} \ \Sigma_0,
        \end{cases}
    \end{align*}
    has a unique solution $w \in C^\infty(D^+(\Sigma_0))$. Here $\overline{\overline{\nabla}}= \widehat{\nabla}(\cdot \circ \Psi) \circ \Psi^{-1}$. The solution $w$ can be decomposed into $w=w^* + \phi$ such that $\phi$ has the following properties
    \begin{itemize}
        \item $\phi \circ \Psi \in C^\infty([0,\infty) \times \overline{\B^d}; W)$,
        \item $\phi \circ \Psi(s,\cdot)$ is radial for all $s \geq 0$.
    \end{itemize}
    The solution is attracted by $w^*$ in the sense that there exists $\delta >0$ such that
    \begin{align*}
        \sup_{s \geq 0}\left(e^{\delta s}\sqrt{\norm{(w^*- w)\circ \Psi(s,\cdot)}{H^k\left(\B^d;W\right)}^2+\norm{(\overline{\overline{\nabla}}w^*- \overline{\overline{\nabla}}w) \circ \Psi(s,\cdot)}{H^{k-1}\left(\B^d;W\right)}^2}\right) \\
        \lesssim \sqrt{\norm{f_1 \circ \Psi(0,\cdot)}{H^k\left(\B^d;W\right)}^2+ \norm{f_2 \circ \Psi(0,\cdot)}{H^{k-1}\left(\B^d;W\right)}^2}.
    \end{align*}
\end{corollary}
At last, consider $u(t,x):= w(t,x)/t$. Adjusting the data from $C_{\operatorname{rad}}^\infty(\overline{\B^d};W)^2$ to $C_{\operatorname{rad}}^\infty(\overline{\B^d};W_1)^2$ due to $t= \frac{1}{1-\abs{y}^2}$ on $\Sigma_0$, and showing that the second evolution component can be replaced by a rescaled time derivative yields \cref{mainthm}.
\subsection{Proof of \cref{mainthm}}
\begin{proof}
    Throughout this proof denote functions obtained by concatenation with the coordinate transform $\Psi$ by $(\cdot)_\Psi$. Setting $u(t,x)= w(t,x)/t$ in the previous corollary yields that $u$ solves
    \begin{align*}
        \begin{cases}
            0= \square u(t,x)+ (d-3) \frac{\sin(2\abs{x}u(t,x))-2\abs{x}u(t,x)}{2 \abs{x}^3} \ \ &\text{in} \ D^{+}(\Sigma_0) \setminus \Sigma_0,\\
            u= u^* + g_1  &\text{in} \ \Sigma_0,\\
            \overline{\nabla}u = \overline{\nabla}u^* + g_2 &\text{in} \ \Sigma_0,
        \end{cases}
    \end{align*}
    where we defined $\overline{\nabla}_{t,x}:= \frac{1}{t}\overline{\overline{\nabla}}_{t,x}(t \cdot), g_1(t,x) := f_1(t,x)/t, g_2(t,x)= f_2(t,x)/t$. We see that assuming $((g_1)_\Psi(0,\cdot),(g_2)_\Psi(0,\cdot)) \in C_{\operatorname{rad}}^\infty(\overline{\B^d};W_1)^2$ is equivalent to assuming $((f_1)_\Psi(0,\cdot),(f_2)_\Psi(0,\cdot)) \in C_{\operatorname{rad}}^\infty(\overline{\B^d};W)^2$. The uniqueness in the space of smooth functions stays true, since multiplication by $t$ is a diffeomorphism on functions defined on $D^{+}(\Sigma_0)$. Putting $\psi:= u^*-u$  in the stability statement with the fact that $t= \frac{e^s}{1-\abs{y}^2}$ on $\Sigma_s$ and changing the norms accordingly shows that
    \begin{align*}
        \sup_{s \geq 0} \left(e^{(1+\delta) s} \norm{\left(\psi_\Psi(s,\cdot), (\overline{\nabla}\psi)_\Psi(s,\cdot)\right)}{H^k\times H^{k-1}(\B^d;W_1)} \right)
        \lesssim \norm{\left((g_1)_\Psi(0,\cdot), (g_2)_\Psi(0,\cdot)\right)}{H^{k} \times H^{k-1}\left(\B^d;W_1\right)}.
    \end{align*}
    The specific form
    \begin{align*}
         (\overline{\nabla}\psi)_\Psi(s,y)= \left(\pt_s + y^j \pt_j + \frac{d-1}{1-|y|^2}\right)\psi_\Psi(s,y)
    \end{align*}
    together with the inequality
    \begin{align*}
        \norm{\left((\cdot)^j \pt_j +\frac{d-1}{1-|\cdot|^2}\right)\psi_\Psi(s,\cdot)}{H^{k-1}(\B^d;W_1)} =\norm{(\cdot)^j \pt_j\left(W_1 \psi_\Psi(s,\cdot)\right)}{H^{k-1}(\B^d)} \lesssim \norm{W_1 \psi_\Psi(s,\cdot)}{H^k(\B^d)}
    \end{align*}
    implies $\norm{(\psi_\Psi(s,\cdot),\pt_s \psi_\Psi(s,\cdot))}{H^k \times H^{k-1}(\B^d;W_1)} \simeq \norm{(\psi_\Psi(s,\cdot),(\overline{\nabla} \psi)_\Psi(s,\cdot))}{H^k \times H^{k-1}(\B^d;W_1)}$. Consider the time derivative $\pt_t \psi(t,x)$ and rescale with $\omega(t,x):=\frac{t^2-|x|^2}{t}$. Note that on $\Sigma_0$ we obtain $\omega \equiv 1$. In FHSC we have $(\omega \pt_0 \psi)_\Psi(s,y)=\left((1+|y|^2)\pt_s - (1-|y|^2) y^j \pt_j \right) \psi_\Psi(s,y)$. Similar to before we obtain the estimate 
    \begin{align*}
        \norm{(1-|\cdot|^2)(\cdot)^j \pt_j \psi_\Psi(s,\cdot)}{H^{k-1}(\B^d;W_1)} \lesssim \norm{\psi_\Psi(s,\cdot)}{H^k(\B^d;W_1)}.
    \end{align*}
    Together with $1+|\cdot|^2 \simeq 1$ we get
    \begin{align*}
        \norm{(\psi_\Psi(s,\cdot),(\omega \pt_0\psi)_\Psi(s,\cdot))}{H^k \times H^{k-1}(\B^d;W_1)} \simeq \norm{(\psi_\Psi(s,\cdot),(\overline{\nabla} \psi)_\Psi(s,\cdot))}{H^k \times H^{k-1}(\B^d;W_1)}.
    \end{align*}
    This justifies replacing $\overline{\nabla}$ with $\omega \pt_0$ in the formulation of the initial value problem and the stability estimate, resulting in the first part of \Cref{mainthm}. 
    
    The additional statement about the decay rate is then deduced as follows. Sobolev embedding implies
    \begin{align*}
        \norm{W_1(\cdot)\psi_\Psi(s,\cdot)}{L^\infty \left(\B^d\right)} \lesssim e^{-(1+\delta) s} \ \text{for all} \ s \geq 0.
    \end{align*}
As a consequence (note that $W_1 \geq 1$), we have
    \begin{align*}
        \abs{u^*_\Psi(s,y)-u_\Psi(s,y)} \lesssim e^{-(1+\delta) s} 
    \end{align*}
    for all $s \geq 0$ and all $y \in \B^d$. In Cartesian coordinates this reads
    \begin{align*}
        \abs{u^*(t,x)-u(t,x)} \lesssim \left(\frac{t}{t^2-|x|^2}\right)^{1+\delta}
    \end{align*}
    for all $(t,x) \in D^+(\Sigma_0)$.
    Consider a fixed $x \in \R^d$. We have $\abs{u^*(t,x)-u(t,x)} \lesssim t^{-(1+\delta)}$ for all $t \geq 1$ such that $(t,x) \in D^+(\Sigma_0)$. Since $u^*(t,x) \simeq t^{-1}$ for all $t \geq 1$ such that $(t,x) \in D^+(\Sigma_0)$, we deduce that the same is true for $u(\cdot,x)$. This completes the proof of \Cref{mainthm}.
\end{proof}

%% file: finitespeedofproppic.tex
\begin{figure}
    \centering
\begin{minipage}[t]{0.45\textwidth}
        \begin{tikzpicture}
        \begin{axis}[
            axis lines=middle,
            xlabel={$|x|$},
            ylabel={$t$},
            xmin=0, xmax=1.1,
            ymin=0.9, ymax=2,
            ticks=none
        ]
            \addplot[black, very thick, domain=0:1, samples=200] 
                {(1+sqrt(4*x^2+1))/2};
            \addlegendentry{$l(y)$}

            \addplot[black, domain=0:1, samples=200] 
                {2.5*abs(x) > 2*x^2/(-1 + sqrt(1 + 4*x^2)) ? 2.5*abs(x) : nan};
            \addlegendentry{$2.5|y|$}
            
            \addplot[black, domain=0:1, samples=201] 
                {5*abs(x) > 2*x^2/(-1 + sqrt(1 + 4*x^2)) ? 5*abs(x) : nan};
            \addlegendentry{$5|y|$}

            \addplot[black, domain=0:1, samples=200] 
                {(sqrt(1 + 4*x^2*0.875^2) + 1)/(2*0.875)};
            \addlegendentry{$k(y, 0.55)$}
            
            \addplot[black, domain=0:1, samples=200] 
                {(sqrt(1 + 4*x^2*0.75^2) + 1)/(2*0.75)};
            \addlegendentry{$k(y, 0.75)$}
    
            \addplot[black, domain=0:1, samples=200] 
                {(sqrt(1 + 4*x^2*0.625^2) + 1)/(2*0.625)};
            \addlegendentry{$k(y, 0.75)$}

            \addplot[red, very thick, domain=0:1, samples=200] 
                {(sqrt(1 + 4*x^2*0.875^2) + 1)/(2*0.875) < 3/2-abs(1/2-x) ? (sqrt(1 + 4*x^2*0.875^2) + 1)/(2*0.875) : nan};
            \addlegendentry{$k(y, 0.55)$}
    
            \addplot [name path= B, red, very thick, domain=0:1, samples=200] 
                {(sqrt(1 + 4*x^2*1^2) + 1)/(2*1) < 3/2-abs(1/2-x) ? (sqrt(1 + 4*x^2*1^2) + 1)/(2*1) : nan};
            \addlegendentry{$k(y, 0.55)$}
    
            \addplot [name path= A, dashed, domain=0:1, samples=200]
                {3/2-abs(1/2-x) > 2*x^2/(-1 + sqrt(1 + 4*x^2)) ? 3/2-abs(1/2-x) : nan};
    
            \addplot[gray, opacity= 0.3] fill between [of= A and B];
    
            \node at (0.5,1.5) [circle,scale=0.3, draw=black, fill=black]{};
            \node[black, above] at (0.5,1.5) {\small$\Psi(s_0,y_0)$};
    
            \node[black,right] at (1,1.72) {\small$\Sigma_{s'}$};
            \node[black,right] at (1,1.62) {\small$\Sigma_0$};
            
            \legend{};
        \end{axis}
    \end{tikzpicture}
    \captionsetup{justification=raggedright, width=\linewidth}
    \captionof{figure}{Projection of the causal past of the point $(t,x)=\Psi(s_0,y_0)$ up to the initial hyperboloid. The red patches of the hyperboloids $\Sigma_{s'}, \Sigma_0$ are the images of $D_{s'}$ resp. $D_0$ for a given $0 \leq s' < s_0$.}
        \label{fsoppic1}
\end{minipage}
\hfill
\begin{minipage}[t]{0.45\textwidth}
    \begin{tikzpicture}[declare function= {g(\x)=(\x <= 1/3) * (ln(1+\x)) + (\x > 1/3) * (ln(2-2*\x));}]
        \begin{axis}[
        axis lines= middle,
        xlabel= {$|y|$},
        ylabel= {$s$},
        xmin= 0, xmax=2/3,
        ymin= 0, ymax=1/2,
        ticks=none
        ]

        \addplot[black, domain=0:1, samples=2]
        {ln(8/7)};
        \addplot[black, domain=0:1, samples=2]
        {ln(8/6)};
        \addplot[black, domain=0:1, samples=2]
        {ln(8/5)};
        \draw[black] (10/50,0)--(10/50,1/2);
        \draw[black] (10/25,0)--(10/25,1/2);
        \node at (1/3,0.28768207245) [circle, scale=0.3,draw=black, fill=black]{};
        % this weird number is an approx. of -log(3/4) %
        \node[black,above] at (1/3,0.28768207245) {\small $(s_0,y_0)$};

        \addplot[name path=C, dashed, domain= 0:1/2, samples=400]
            {g(x)};

        \addplot[name path=D, red, very thick, domain= 0:1/2, samples=2]
        {0};

        \addplot[gray, opacity=0.3] fill between [of= C and D];
        \addplot[red, very thick, domain= 0:1/2, samples= 200]
        { ln(8/7) < g(x) ? ln(8/7) : nan};
        % log(8/7) is the s for the hyperboloid in the upper picture %
        \node[red,above] at (7/24,0.13353139262) {\small $D_{s'}$};
        \node[red,above] at (7/24,0) {\small $D_0$};
        % 0.13353139262 approximates log(8/7)%
        \end{axis}
    \end{tikzpicture}
    \captionsetup{justification=raggedright, width=\linewidth}
    \captionof{figure}{Projection of $J^{-}(s_0,y_0)$, i.e. left picture in FHSC. Note that the boundary (dashed) is not given by a straight line but by two logarithmic functions.}
    \label{fsoppic2}
\end{minipage}
\end{figure}